\documentclass[11pt, reqno, a4paper]{amsart}
\usepackage{orcidlink}
\usepackage[margin=1.3in]{geometry}
\usepackage{setspace, enumitem, xcolor}
\usepackage{hyperref}
\usepackage{float}
\usepackage{hyperref}
\usepackage[nameinlink]{cleveref}

\usepackage{amsfonts,amsmath,amssymb,amsthm, mathrsfs, todonotes}
\usepackage{mathtools}
\usepackage{parskip, enumitem} 
\newtheorem{theorem}{Theorem}[section]
\newtheorem{lemma}[theorem]{Lemma}
\newtheorem{proposition}[theorem]{Proposition}
\newtheorem{question}[theorem]{Question}

\theoremstyle{definition}
\newtheorem{definition}[theorem]{Definition}
\newtheorem{example}[theorem]{Example}

\newtheorem{corollary}[theorem]{Corollary}
\newtheorem*{definition*}{Definition}

\newtheorem{remark}[theorem]{Remark}

\crefname{question}{question}{questions}
\Crefname{question}{Question}{Questions}
\crefname{mainthm}{theorem}{theorems}
\Crefname{mainthm}{Theorem}{Theorems}

\newcommand{\calB}{{\mathcal{B}}}

\newcommand{\calF}{{\mathcal{F}}}

\newcommand{\calH}{{\mathcal{H}}}

\newcommand{\calM}{{\mathcal{M}}}

\newcommand{\calT}{{\mathcal{T}}}

\newcommand{\calV}{{\mathcal{V}}}

\newcommand{\scrD}{{\mathscr D}}

\newcommand{\scrG}{{\mathscr G}}

\newcommand{\scrT}{{\mathscr T}}

\newcommand{\ip}[1]{\left\langle #1 \right\rangle}
\newcommand{\norm}[1]{\left\lVert #1 \right\rVert}
\newcommand{\abs}[1]{\left\lvert #1 \right\rvert}

\renewcommand{\le}{\leqslant}
\renewcommand{\ge}{\geqslant}
\renewcommand{\leq}{\leqslant}
\renewcommand{\geq}{\geqslant}

\newcommand{\R}{\mathbb R}
\newcommand{\N}{\mathbb N}
\newcommand{\Z}{\mathbb Z}

\newcommand{\C}{\mathbb C}

\newcommand{\Chi}{\mathsf{Chi}}
\newcommand{\pr}{\mathsf{par}}

\numberwithin{equation}{section}

\usepackage{thmtools}
\usepackage{thm-restate}

\begin{document}

\title[Wold-type decompositions for weighted shifts on directed trees]{Wold-type decompositions for norm-increasing $m$-concave weighted shifts on directed trees}
\author[Ashish Kujur]{Ashish Kujur \orcidlink{0009-0001-5952-7291}}
\address{School of Mathematics, Indian Institute of Science Education and Research, Thiruvananthapuram}
\email{ashishkujur23@iisertvm.ac.in}
\thanks{The author is supported through the Senior Research Fellowship of the Council of Scientific and Industrial Research (CSIR), India (Ref. No. 09/0997(18166)/2024-EMR-I)}

\begin{abstract}
    In this article, we show that every norm-increasing $3$-concave weighted
    shift on a directed tree admits Wold-type decomposition. We then classify the $4$-isometries within a class of weighted shifts on a rootless directed tree introduced by S. Chavan and S. Trivedi, and use this classification to construct analytic norm-increasing strict $4$-isometries without the wandering subspace property. This answers a question of S. Shimorin in the negative for every $m\in \mathbb{Z}_{\ge 4}$.
\end{abstract}

\maketitle

\section{Introduction}
    For $x \in \R$, we define $\mathbb{Z}_{\ge x} := \{ n \in \Z : n \ge x \}$. Let $\calH$ be a separable Hilbert space, and we denote by $\calB (\calH)$ the $C^{*}$-algebra of all bounded linear operators on $\calH$. Let $\calM$ be a subspace of $\calH$ and $T \in \calB (\calH)$. We define
    \begin{equation*}
        [\calM]_{T} := \bigvee \left\{ T^{n} x : x \in \calM, \, n \in \Z_{\ge 0} \right\},
    \end{equation*}
    where $\bigvee$ denotes the closed linear span. Then $[\calM]_{T}$ is the smallest closed subspace of $\calH$ containing $\calM$ and invariant under $T$. If there is no ambiguity, we simply write $[\mathcal{M}]$ instead of $\left[ \mathcal{M} \right]_{T}$. We will primarily deal with the case $\calM = \ker T^{*}$.

    We say that \textit{$T$ possesses the wandering subspace property} if $[\ker T^{*}]_T = \calH$ and say \textit{$T$ is analytic} if
    \begin{equation*}
        \calH_{\infty} (T) := \bigcap_{n \in \Z_{\ge 0}} T^n (\calH) = \{ 0 \}.
    \end{equation*}

    The classical Wold decomposition theorem (see \cite[Theorem 1.1]{MR2760647} and \cite[\S 1.3]{MR822228}) states that every isometry
    $T \in \calB(\calH)$ admits the orthogonal decomposition $\calH = \calH_{\infty} (T) \oplus [\ker T^{*}]$, where $\calH_{\infty} (T)$ is a reducing subspace for $T$ and the operator $T\mid_{\calH_{\infty} (T)}$ is unitary on $\calH_{\infty} (T)$. Moreover, for an isometry $T$, $[\ker T^*] = \bigoplus_{n \in \Z_{\ge 0}} T^n (\ker T^{*})$ where we use $\bigoplus$ to denote orthogonal direct sums of closed subspaces of $\calH$.
    
    Motivated by the classical Wold decomposition, S. Shimorin introduced the notion of a \textit{Wold-type decomposition} (see \cite[Definition 1.1]{zbMATH01572594}). We say that an operator $T$ \textit{admits Wold-type decomposition} if $\calH_{\infty} (T)$ is a reducing subspace for $T$, the  operator $T\mid_{\calH_{\infty} (T)}$ is unitary and $\calH = \calH_{\infty} (T) \oplus [\ker T^{*}]$. In particular, the classical Wold decomposition theorem shows that every isometry admits Wold-type decomposition. Moreover, if an operator $T \in \calB (\calH)$ admits the Wold-type decomposition, then it is analytic if and only if it possesses the wandering subspace property.
    
    Building on the work of S. Richter (see \cite[Theorem 1]{MR936999}), S. Shimorin showed that if $T$ is a bounded operator satisfying either 
    \begin{equation*}
        \norm{T^{2}x}^{2} + \norm{x}^{2} \le 2 \norm{Tx}^{2}, \qquad x \in \calH,
    \end{equation*}
    or 
    \begin{equation*}
        \norm{Tx + y}^{2} \le 2 (\norm{x}^2 + \norm{Ty}^2)\qquad x,y \in \calH, 
    \end{equation*}
    then $T$ admits Wold-type decomposition (see \cite[Theorem 3.6]{zbMATH01572594}). 
    
    For $m \in \Z_{\ge 1}$ and $T\in \calB (\calH)$, we define
    \begin{equation}
        \label{eq:beta-m}
        \beta_{0} (T) = I, \qquad\beta_m (T) := \sum_{j=0}^{m} (-1)^{m-j} \binom{m}{j} T^{*j} T^j.
    \end{equation}
    We say that $T$ is a \textit{$m$-isometry} if $\beta_m (T) = 0$ and say $T$ is \textit{$m$-concave} if $\beta_m (T) \le 0$. We say that $T$ is \textit{norm-increasing} if $\beta_1(T)=T^*T-I\ge 0$. 
    
    The notion of $m$-isometry was introduced by J. Agler in \cite{MR1037599} and was subsequently studied by him and M. Stankus in \cite{Agler1995, MR1346617, MR1382018}. The class of $2$-concave operators, also simply known as concave operators, was introduced by S. Richter in his work on the Dirichlet shift \cite{MR936999, MR1013337} and its higher-order counterparts, namely, $m$-concave operators were introduced by S. Shimorin in \cite{zbMATH01572594}.
    
    As mentioned above, S. Shimorin was able to show that concave operators (that is, operators $T \in\calB (\calH)$ corresponding to the inequality $\beta_2 (T) \le 0$) admit Wold-type decompositions. He also showed that if $T\in\calB(\calH)$ is a norm-increasing $m$-concave operator then $\calH_\infty(T)$ reduces $T$, and the restriction $T|_{\calH_\infty(T)}$ is unitary (see \Cref{prop:Shimorin-hyper-range} and \cite[Proposition 3.4]{zbMATH01572594}). In view of these results, S. Shimorin posed the following question:

    \begin{question}[{\cite[p. 185]{zbMATH01572594}}]
        Does every norm-increasing $m$-concave operator admit a Wold-type decomposition for $m\in \Z_{\ge 3}$?
        \label{qn:Shimorin}
    \end{question}

    By \cite[Proposition~3.4]{zbMATH01572594} (see also \Cref{prop:Shimorin-hyper-range}), \Cref{qn:Shimorin} is equivalent, for each fixed $m\in\Z_{\ge 3}$, to the following question:

    \begin{question}
        Does every analytic norm-increasing $m$-concave operator possess the wandering subspace property for $m\in \Z_{\ge 3}$?
        \label{qn:Shimorin-analytic-version}
    \end{question}

    In \cite[Example 3.1]{MR4078097}, A. Anand, S. Chavan and S. Trivedi constructed a family of analytic cyclic $3$-isometries which are norm-increasing on the orthogonal complement of a one dimensional subspace but do not possess the wandering subspace property. Their examples are weighted shifts on one-circuit directed graphs associated with rooted directed trees. They further showed that, in their setting, a counterexample to \Cref{qn:Shimorin} cannot be obtained from a weighted shift on a one-circuit directed graph (see \cite[Corollary~4.4]{MR4078097}). This, however, does not resolve \Cref{qn:Shimorin}, as these examples are norm-increasing only on the orthogonal complement of a one-dimensional subspace, rather than on all of $\calH$.

    Motivated by the work of S. Chavan and S. Trivedi in \cite{MR5052237}, we investigate \Cref{qn:Shimorin} in the setting of weighted shifts on directed trees. It is known that every bounded weighted shift on a rooted directed tree admits Wold-type decomposition (see \cite[Proposition 1.3.4]{MR3740250}). However, this is no longer true for rootless directed trees. In fact, S. Chavan and S. Trivedi studied this problem in \cite{MR5052237} and obtained a characterisation of left-invertible weighted shifts on rootless directed trees with positive weights which admit Wold-type decomposition (see \cite[Theorem 1.3]{MR5052237}).

    In our first main result, we show that every norm-increasing $3$-concave weighted shift on a directed tree, rooted or rootless, admits a Wold-type decomposition. We refer the reader to \Cref{sec:preliminaries} for the preliminaries and the notations used for the rest of the article.

    \begin{restatable}{mainthm}{threeConcave}
        Let $\scrT=( V, E)$ be a directed tree, and let $S_{\lambda}\in\calB(\ell^2( V))$ be a norm-increasing $3$-concave weighted shift with complex weights. Then $S_{\lambda}$ admits Wold-type decomposition.
        \label{thm:3-concave}
    \end{restatable}

    Hence, \Cref{qn:Shimorin} has an affirmative answer for $m=3$ in the aforementioned setting. We remark that the norm-increasing hypothesis cannot be dropped: there are analytic, non-cyclic, left-invertible, $3$-isometric weighted shifts on directed trees with positive weights which do not possess the wandering subspace property (see \Cref{ex:analytic-non-cyclic-no-WSP}). Such counterexamples cannot be cyclic, since analytic cyclic weighted shifts with positive weights force the directed tree to be rooted (see \Cref{remark:analytic-finite-codimension-rooted}). We are, however, unable to answer whether every norm-increasing $3$-concave operator on a Hilbert space admits Wold-type decomposition.

    We now consider the higher-order case, that is, $m \in \Z_{\ge 4}$. In \cite[Proposition 5.1]{MR5052237}, S. Chavan and S. Trivedi constructed a family of analytic, norm-increasing weighted shifts which do not admit Wold-type decompositions. These weighted shifts are constructed on the \textit{rootless quasi-Brownian directed tree of valency $2$} (see \cite[Example 5.1(c)]{MR4019095}). See \Cref{fig:quasi-brownian-2-intuit} for an illustration of this tree, where the branching vertices are filled in. 

    \begin{figure}[H]
        \centering  
        \includegraphics[width=0.28\textwidth]{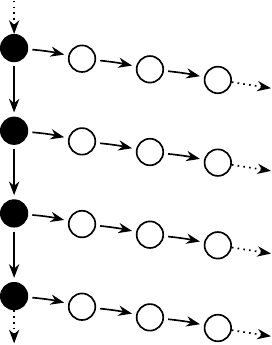}
        \caption{The rootless quasi-Brownian tree of valency $2$}
        \label{fig:quasi-brownian-2-intuit}
    \end{figure}

    As mentioned earlier, every bounded weighted shift on a rooted directed tree admits a Wold-type decomposition. Therefore, within the setting of directed trees, a counterexample to \Cref{qn:Shimorin} must arise on a rootless directed tree. The tree must also be leafless, since  otherwise the weighted shift cannot be injective and, consequently, cannot be norm-increasing (see \cite[Proposition 3.1.7]{MR2919910}).  

    To obtain a counterexample to \Cref{qn:Shimorin-analytic-version}, one must look for leafless, rootless trees which possess analytic weighted shifts. For positive weights, an obstruction for having analytic weighted shifts is the existence of a \textit{generalised root}, a notion which was introduced in \cite{MR3532172} (see \Cref{def:generalised-root}).
    
    We note in \Cref{prop:generalised-root-not-analytic} that no weighted shift with positive weights on a tree admitting a generalised root can be analytic. Moreover, a rootless, leafless directed tree with a branching vertex admits a generalised root whenever its \textit{branching index is finite} or or one of its \textit{generations} is finite (see \Cref{prop:characterisation-of-finite-branching-rootless-trees,prop:some-finite-generation-implies-generalised-root}). The rootless quasi-Brownian directed tree of valency $2$ is leafless, has infinite branching index, and has every generation infinite. Hence, it is a natural setting for constructing counterexamples to \Cref{qn:Shimorin-analytic-version}, and hence to \Cref{qn:Shimorin}.
    
    For our second main result, we classify all the $4$-isometries within the class of weighted shifts introduced by S. Chavan and S. Trivedi (see \Cref{thm:classify-4-isometries}). We use this classification to produce analytic, norm-increasing $4$-isometries which do not possess the wandering subspace property. Since an analytic operator admits Wold-type decomposition if and only if it possesses the wandering subspace property, these operators do not admit Wold-type decompositions. Consequently, our second main result answers \Cref{qn:Shimorin} in the negative for $m \in \Z_{\ge 4}$.

    \begin{restatable}{mainthm}{mConcaveCounterexample}
    \label{thm:counterexample}
        For every $m\in \Z_{\ge 4}$, there is a family of analytic, norm-increasing $m$-concave weighted shift on the rootless quasi-Brownian directed tree of valency $2$ with positive weights which do not admit a Wold-type decomposition.
    \end{restatable}

    As mentioned earlier, the operators in \Cref{thm:counterexample} may be chosen to be $4$-isometries. Together with \Cref{thm:3-concave}, this shows a contrast between the cases $m=3$ and $m\ge4$ for
    norm-increasing $m$-concave weighted shifts on directed trees: the Wold-type decomposition always holds when $m\in \{ 1,2,3 \}$, whereas it may fail for every $m\in \Z_{\ge 4}$.

    The paper is organized as follows. In \Cref{sec:preliminaries}, we recall the necessary preliminaries on directed trees and weighted shifts on directed trees. In \Cref{sec:positive-theorem}, we prove \Cref{thm:3-concave}. In \Cref{sec:generalised-root}, we discuss sufficient conditions for the existence of a generalised root. In \Cref{sec:4-isometry}, we study the class of weighted shifts considered in \cite{MR5052237}, classify the 4-isometries within this class, and use this classification to establish \Cref{thm:counterexample}. In \Cref{sec:3-isometry}, we produce examples of analytic $3$-isometries on the quasi-Brownian tree of valency $2$ which do not possess the wandering subspace property, are not norm-increasing and are not cyclic. In \Cref{sec:last}, we state some open problems.

\section{Preliminaries}\label{sec:preliminaries}

We remind the reader again that throughout the paper $\mathbb Z_{\geq x}:=\{n\in\mathbb Z:n\geq x\}$ for $x\in\mathbb R$. All Hilbert spaces are complex, and $\mathcal B(\mathcal H)$ denotes the $C^*$-algebra of bounded linear operators on a Hilbert space $\mathcal H$.

\subsection{Left-invertible and \texorpdfstring{$m$}{m}-concave operators}  We say that an operator $T\in\calB(\calH)$ is \textit{left-invertible} if it has a bounded left inverse, or equivalently, if it is bounded below.

We remind the reader again that an operator $T\in \calB (\calH)$ is called a \textit{$m$-isometry} if $\beta_m (T) = 0$ and is called \textit{$m$-concave} if $\beta_m (T) \le 0$ (see \cref{eq:beta-m}). Moreover, we say that $T$ is \textit{norm-increasing} if $\beta_1(T)=T^*T-I\ge 0$. Since every norm-increasing operator is bounded below, it is left-invertible. If $T$ is left-invertible, then $T^* T$ is invertible and we define the \textit{Cauchy dual $T'$ of $T$} by $T' := T (T^* T)^{-1}$. Moreover, we denote the \textit{Moore--Penrose inverse of $T$} by $L$. Then $L= T'^{*} = (T^* T)^{-1}T^*$ and $(T')' = T$.

We note the following identity:
\begin{equation}
    \beta_{m+1} (T) = T^* \beta_m (T) T -\beta_m (T), \qquad m \in \Z_{\ge 0}, \,  T \in \calB (\calH).
    \label{eq:beta-m-recursion}
\end{equation}
This identity shows that if $T \in \calB (\calH)$ is a $m$-isometry then it is a $(m+1)$-isometry. We call any $m$-isometry which is not a $(m-1)$-isometry a \textit{strict $m$-isometry}. We remark that the proof of the above identity follows simply by induction. 

We state a lemma which will be useful later to establish \Cref{thm:3-concave}.

\begin{lemma}[{\cite[Lemma 5.1]{MR4534903}}]
    Let $T$ be a bounded $m$-concave operator on a Hilbert space $\calH$ for some $m \in \Z_{\ge 2}$. Then 
    \begin{equation*}
        (T^{*})^{n} T^{n} \le \sum_{j =0}^{m-1} \binom{n}{j} \beta_j (T), \qquad n \in \Z_{\ge 0}.
    \end{equation*}
    Moreover, $\beta_{m-1} (T) \ge 0$.
    \label{lem:m-concave}
\end{lemma}

    In \cite[Lemma 5.1]{MR4534903}, the previous inequality is stated for $n \in \Z_{\ge m}$. If $n \in \Z_{\ge 0}$ and $n < m$, we have
        \begin{equation*}  
            T^{*n}T^n = \sum_{j=0}^{n}\binom nj\beta_j(T) = \sum_{j=0}^{m-1}\binom nj\beta_j(T), 
        \end{equation*} 
    where $\binom{n}{j}=0$ for $j \in \Z_{\ge n+1}$. Thus, for $n \in \Z_{\ge 0}$ and $n < m$, the inequality holds as an equality. This justifies the formulation of \Cref{lem:m-concave} for every $n\in\Z_{\ge 0}$. For a proof of \Cref{lem:m-concave}, see \cite[Theorem 2.5 and Corollary 2.4]{MR3314868}.

    We state three results due to S. Shimorin that will be used later.

    \begin{proposition}[{\cite[Proposition~3.4]{zbMATH01572594}}]
    Let $m \in \Z_{\ge 2}$, and let $T\in\calB(\calH)$ be a norm-increasing $m$-concave operator. Then $\calH_\infty(T)$ reduces $T$, and the restriction $T|_{\calH_\infty(T)}$ is unitary.
    \label{prop:Shimorin-hyper-range}
    \end{proposition}

    \begin{proposition}[{\cite[Proposition 2.10]{zbMATH01572594}}]
        Let $T$ be a left-invertible operator. If $\calH_{\infty} (T)$ is reducing for $T$ then $\calH_{\infty} (T) \subset \calH_{\infty} (T')$.
        \label{prop:hyperange-T-subset-hyperange-T'}
    \end{proposition}

    \begin{proposition}[{\cite[Corollary 2.8]{zbMATH01572594}}] 
    \label{prop:analytic-wsp-in-terms-of-cauchy-dual}
    Let $T$ be a left-invertible operator. The following statements hold:
    \begin{enumerate}
        \item[\rm (i)] $T$ is analytic if and only if $T'$ possesses the wandering subspace property.
        \item[\rm (ii)] $T'$ is analytic if and only if $T$ possesses the wandering subspace property.
    \end{enumerate}
    \end{proposition}
    
    We remark that S. Shimorin states \Cref{prop:Shimorin-hyper-range,prop:hyperange-T-subset-hyperange-T'} in \cite[Proposition~3.4]{zbMATH01572594} and \cite[Proposition 2.10]{zbMATH01572594} respectively in terms of the Cauchy dual $T'$. The formulation above is equivalent and is obtained by using the fact that $(T')' = T$ for any left-invertible operator $T \in \calB (\calH)$.

\subsection{Directed Trees}
We closely follow \cite[Section 2.1]{MR2919910} and \cite[Section 1.3]{MR3740250}. A \textit{directed graph} is a pair $\scrG = ( V,  E)$ where $V \ne \emptyset$ and $ E \subset (V \times V) \setminus \{ (v,v) : v \in V \}$. An element of $ V$ is called a \textit{vertex of $\scrG$} and an element of $ E$ is called an \textit{edge of $\scrG$}.

Let $\scrG = ( V,  E)$ be a directed graph. If $W \subset V$ and $W \ne \emptyset$, then the pair $\scrG_W =(W, (W \times W) \cap E)$ is called a \textit{subgraph of $\scrG$}. Define $\tilde{ E} = \{ \{ u,v \} : (u,v) \in  E \text{ or } (v,u) \in  E \}$. We call $\tilde{\scrG} = ( V, \tilde{ E})$ the \textit{undirected graph associated with $\scrG$}. A finite sequence $\{ u_j \}_{j=1}^{n}$ in $ V, \quad n \in \Z_{\ge 2}$ is called an \textit{undirected path joining $u,v\in  V$} if $u_1 = u$, $u_n = v$ and $\{ u_j, u_{j+1} \} \in \tilde{ E}$ for each $j \in \Z_{\ge 1}$ with $j \le n-1$. We say \textit{$\scrG$ is connected} if for any two distinct vertices $u,v \in  V$ there is an undirected path joining $u,v \in  V$. A finite sequence $\{u_j\}_{j=1}^{n}$ of distinct vertices, with $n\in\Z_{\ge2}$, is called a \textit{circuit} if $(u_j,u_{j+1})\in E$ for $j=1,\ldots,n-1$ and
$(u_n,u_1)\in E$. If for any given vertex $u \in  V$, there is a unique $v \in  V$ such that $(v,u) \in  E$, we say \textit{$v$ is the parent of $u$} and denote it by $\pr (u)=v$. Then $\pr (\cdot)$ is a partial function from $ V$ to $ V$. We write $\pr^{\ip{n}}(u)$ for $\pr$ composed with itself $n$ times at $u$, whenever defined.
  
  A vertex $u\in  V$ is called a \textit{root of $\scrG$} if there is no vertex $v$ such that $(v,u) \in  E$. We denote the set of all roots by $\mathsf{Root} (\scrG)$. If $\mathsf{Root} (\scrG)$ is a singleton, we denote that single element as $\mathsf{root} (\scrG)$, or simply, $\mathsf{root}$ if there is no room for ambiguity. We define $ V ^{\circ} :=  V \setminus \mathsf{Root} (\scrG)$. For each $u \in  V$, we define the set $\Chi (u)$ given by
\begin{equation*}
    \Chi (u) = \{ v \in  V \, : \, (u, v) \in  E \}.
\end{equation*}
If $u,v \in  V$ and $v \in \Chi (u)$ then we say \textit{$v$ is a child of $u$}.

A directed graph $\scrT = ( V,  E)$ is called a \textit{directed tree} if 
\begin{enumerate}[noitemsep]
    \item[(i)] $\scrT$ is connected,
    \item[(ii)] $\scrT$ has no circuits, and 
    \item [(iii)] each vertex $v \in  V ^{\circ}$ has a parent.
\end{enumerate}

Let $\mathscr T = (V,E)$ be a directed tree, and let $W \subset V$ and $W \ne \emptyset$ such that the subgraph $\calT_{W} = (W, (W\times W) \cap E )$ is connected. We will call $\calT_W$ a \textit{subtree of $\scrT$}.

    We define $ V_{\prec} := \{ u \in  V : \lvert \Chi (u) \rvert \ge 2 \}$. If $v \in  V_{\prec}$ then we say $v$ is a \textit{branching vertex of $\scrT$}. We say $v \in  V$ is a \textit{leaf of $\scrT$} if $\Chi (v) = \emptyset$. The \textit{degree} of a vertex $v$, denoted $\deg v$, is defined by $\deg v = |\Chi (v)|$. If a directed tree does not contain any leaves, we call it \textit{leafless}. We define $V' := \{u \in V : \Chi(u) \ne \emptyset\}$. We say that $\scrT$ is \textit{locally finite} if $\Chi(u)$ is finite for every $u \in V$. For any subset $W \subset  V$, we define $\Chi (W) = \bigcup_{u \in W} \Chi (u)$. Define
    \begin{equation*}
        \begin{aligned}
            \Chi ^{\ip{0}} (W) &:= W, \\
            \Chi ^{\ip{n+1}} (W) &:= \Chi (\Chi ^{\ip{n}} (W)), \qquad n \in \Z_{\ge 0}, \\
            \mathsf{Des} (W) &:= \bigcup_{n \in \Z_{\ge 0}} \Chi^{\ip{n}} (W).
        \end{aligned}
    \end{equation*}
    If $w\in V$ and $n\in\Z_{\ge0}$, we simply write $\Chi^{\ip{n}}(w)$ for $\Chi^{\ip{n}}(\{w\})$ and $\mathsf{Des}(w)$ for $\mathsf{Des}(\{w\})$.

    Given any two vertices $u,v\in V$, we say that \textit{$u$ and $v$ are in the same generation} if $\pr^{\ip{n}}(u)=\pr^{\ip{n}}(v)$ for some $n\in\Z_{\ge 0}$ for which both sides are defined. Define a relation $\sim$ on $ V$ by $u\sim v$ if $u$ and $v$ are in the same generation. Then $\sim$ is an equivalence relation, and we call its equivalence classes the \textit{generations of $\scrT$}. For $u\in V$, we denote the generation containing $u$ by $\scrG_u$ and call it the \textit{generation of $u$}.

    Let $l \in \Z_{\ge 2}$. Following \cite[Example 5.1]{MR4019095}, we say that a directed tree $\scrT = (V,E)$ is a \textit{quasi-Brownian directed tree of valency $l$} if
    \begin{enumerate}
        \item[\rm (i)] there is some vertex $u_0 \in V$ such that $\deg u_0 = l$,
        \item[\rm (ii)] every vertex is of degree $1$ or $l$,
        \item[\rm (iii)] every child of a vertex $v$ of degree $1$ is of degree $1$,
        \item[\rm (iv)] for every vertex $u \in V$ with degree $l$, there is exactly one child of degree $l$ and the remaining $l-1$ child vertices are of degree $1$. 
    \end{enumerate}

    For each $l \in \Z_{\ge 2}$, there are exactly two quasi-Brownian trees of valency $l$ up to isomorphism, one of which is rooted and the other is rootless. In this paper, we primarily deal with the rootless quasi-Brownian tree of valency $2$.

\subsection{Weighted Shifts on Directed trees}
Let $\scrT = ( V,  E)$ be a directed tree. We denote by $\ell ^2 ( V)$ the Hilbert space of complex valued square summable functions on $ V$. The set $\{ e_u : u \in  V \}$ is an orthonormal basis for $\ell ^2 ( V)$ where $e_u :  V \to \C$ is given by
\begin{equation*}
    e_u (v) = \begin{dcases}
        1, & v=u \\
        0, & v \ne u.
    \end{dcases}
\end{equation*}

Given a system $\lambda = \{ \lambda_{u} \}_{u \in  V ^{\circ}}$ of complex numbers, we define the \textit{weighted shift operator $S_{\lambda}$ on $\scrT$ with weights $\lambda$} by 
\begin{equation*}
\begin{gathered}
        \scrD (S_{\lambda}) := \{ f\in \ell ^2 ( V) : \Lambda_{\scrT} f \in \ell ^2 ( V) \}, \\
    S_{\lambda} f = \Lambda_{\scrT} f, \qquad f \in \scrD (S_{\lambda})
\end{gathered}
\end{equation*}
where 
\begin{equation*}
    \left( \Lambda_\scrT f \right) (v) = \begin{dcases}
        \lambda_v \cdot f(\pr (v)), & v\in  V ^{\circ}, \\
        0 & v=\mathsf{root}.
    \end{dcases}
\end{equation*}

We call $S_{\lambda}$ a \textit{weighted shift on $\scrT$ with nonnegative weights} if $\lambda_v\geq 0$ for each $v\in V^{\circ}$, and a \textit{weighted shift on $\scrT$ with positive weights} if $\lambda_v>0$ for each $v\in V^{\circ}$. Unless specified otherwise, a weighted shift on $\scrT$ is understood to have complex weights. When there is a possibility of ambiguity, we explicitly use the term \textit{weighted shift on $\scrT$ with complex weights}.

In this article, we deal with the case where the weighted shifts are only bounded. In such a case, $\mathscr D(S_{\lambda }) = \ell ^2( V)$. Conversely, if $\mathscr D(S_\lambda)=\ell^2( V)$, then $S_\lambda$ is bounded. We recall that if $S_{\lambda}$ is bounded, then 
\begin{equation*}
    S_{\lambda}e_u = \sum_{v\in \Chi (u)} \lambda_v e_v, \qquad \norm{S_{\lambda} e_u}^{2} = \sum_{v\in \Chi (u)} \abs{\lambda_v}^{2}, \qquad u \in  V.
\end{equation*}
Conversely, if $\sup_{u \in  V} \norm{S_{\lambda} e_u}^{2} < \infty$ then $S_{\lambda}$ is bounded. See \cite[Proposition 3.1.3 and Proposition 3.1.8]{MR2919910} for proofs of these facts.

If $W \subset V$, we treat $\ell^2(W)$ as a closed linear subspace of $\ell^2(V)$ by identifying each $f\in\ell^2(W)$ with its extension to $V$ which vanishes on $V\setminus W$. If $S_{\lambda}$ is left-invertible, then we denote its Moore--Penrose inverse by $L_{\lambda}$, that is, $L_{\lambda} = (S_{\lambda}^{*} S_{\lambda})^{-1} S_{\lambda}^{*}$. A bounded weighted shift $S_{\lambda}$ is called \textit{balanced} if $\norm{S_{\lambda} e_u} = \norm{S_{\lambda} e_v}$ whenever $u$ and $v$ belong to the same generation.

Following \cite[Section 6.1]{MR2919910}, for $u\in V$ and $v\in\mathsf{Des}(u)$, we define
\begin{equation*}
\lambda_{u\mid v} = \begin{dcases}
1, & v=u,\\
\prod_{k=0}^{n-1} \lambda_{\pr^{\langle k\rangle}(v)}, & v\in\Chi^{\langle n\rangle}(u),\quad n\in \Z_{\ge 1}.
\end{dcases}
\end{equation*}

Then $\lambda_{u\mid v}=\lambda_{u\mid\pr(v)}\lambda_v$ for $v\in\mathsf{Des}(u)\setminus\{u\}$, and $\lambda_{\pr(v)\mid w}=\lambda_v\lambda_{v\mid w}$ for $v \in V^{\circ}, \; w\in\mathsf{Des}(v)$.

We state some lemmas that will be used in upcoming sections. We sketch a proof of those lemmas which we could not find in the literature.

\begin{lemma}[{\cite[Lemma 6.1.1]{MR2919910}}]
    \label{lem:JJS-powers}
Let $\scrT=( V, E)$ be a directed tree and let $S_\lambda\in\calB(\ell^2( V))$ be a weighted shift on $\scrT$. Then, for every $u\in V$ and $n\in\Z_{\geq0}$,
\begin{equation*}
    S_\lambda^n e_u = \sum_{v\in\Chi^{\langle n\rangle}(u)} \lambda_{u\mid v}e_v, \qquad \norm{S_\lambda^n e_u}^2 = \sum_{v\in\Chi^{\langle n\rangle}(u)} \abs{\lambda_{u\mid v}}^2.
    \end{equation*}
\end{lemma}

\begin{lemma}
    \label{lem:compute-beta-m}
    Let $\scrT = ( V,  E)$ be a directed tree and let $S_{\lambda} \in \calB (\ell ^2 ( V))$ be a weighted shift on $\scrT$. Then
    \begin{equation*}
        \beta_{m} (S_{\lambda}) e_u = \left( \sum_{j=0}^{m} (-1)^{m-j} \binom{m}{j} \norm{S_{\lambda}^{j}e_u }^{2} \right) e_u, \qquad u \in V, m \in \N.
    \end{equation*}
    Consequently, $S_{\lambda}$ is a $m$-isometry, $m \in \N$ if and only if $\beta_{m} (S_{\lambda}) e_u = 0$ for each $u \in  V$.
\end{lemma}
\begin{proof}
 We have by \Cref{lem:JJS-powers} that for every $j\in \Z_{\ge 0}$ and $u\in V$,
\begin{equation*}
S_\lambda^j e_u = \sum_{v\in\Chi^{\langle j\rangle}(u)} \lambda_{u|v}e_v.
\end{equation*}

If $u, v \in  V$ satisfy $u\neq v$, then $\Chi^{\langle j\rangle}(u)$ and $\Chi^{\langle j\rangle}(v)$ are disjoint. Hence, $\langle S_\lambda^j e_u,S_\lambda^j e_v\rangle=0$ for each $j \in \Z_{\ge 1}$ and each $u, v \in  V$ satisfying $u \ne v$. Therefore $S_\lambda^{*j}S_\lambda^j e_u = \norm{S_\lambda^j e_u}^2e_u$ for every $j\in\Z_{\ge0}$ and $u\in V$.

It follows that
\begin{equation*}
\beta_m(S_\lambda)e_u = \left( \sum_{j=0}^{m}(-1)^{m-j}\binom{m}{j} \norm{S_\lambda^j e_u}^2 \right)e_u, \qquad m \in \Z_{\ge 0},\, u \in  V.
\end{equation*}
Since $\{e_u\}_{u\in V}$ is an orthonormal basis of $\ell^2(V)$,
 $S_{\lambda}$ is a $m$-isometry if and only if $\beta_{m} (S_{\lambda}) e_u = 0$ for each $u \in  V$.
\end{proof}

\begin{lemma}
    Let $\scrT = ( V,  E)$ be a rootless directed tree, and let $S_{\lambda}$ be a left-invertible weighted shift on $\ell ^2 ( V)$. Then we have the following for each $u\in  V$ and $n \in \Z_{\ge 0}$,
    \begin{equation*}
        L_{\lambda}^n e_u = \frac{ \overline{\lambda_{\pr^{\langle n\rangle}(u)\mid u}} }{ \prod_{k=1}^n \left\lVert S_\lambda e_{\pr^{\langle k\rangle}(u)} \right\rVert^2 } e_{\pr^{\langle n\rangle}(u)},
    \end{equation*}
    and,
    \begin{equation*}
        \left\lVert S_\lambda^n L_{\lambda}^n e_u \right\rVert ^2 = \frac{ \left|\lambda_{\pr^{\langle n\rangle}(u)\mid u}\right|^2}{ \left( \prod_{k=1}^n \left\lVert S_\lambda e_{\pr^{\langle k\rangle}(u)} \right\rVert^2 \right)^2 } \left\lVert S_\lambda^n e_{\pr^{\langle n\rangle}(u)} \right\rVert^2.
    \end{equation*}
    \label{lem:all-about-L-lambda}
\end{lemma}

\begin{proof}
For each $u \in V$, we have
\begin{equation*}
    L_\lambda e_u = \frac{\overline{\lambda_u}} {\|S_\lambda e_{\pr(u)}\|^2} e_{\pr(u)}
\end{equation*}
Consequently, the first identity follows by induction and the definition of
$\lambda_{\pr^{\langle n\rangle}(u)\mid u}$. The second identity follows from the first one and \Cref{lem:JJS-powers}.
\end{proof}

\section{Norm-increasing \texorpdfstring{$3$}{3}-concave weighted shifts on directed trees admit Wold-type decomposition}\label{sec:positive-theorem}

This section first develops several results needed to prove that every norm-increasing $3$-concave weighted shift on a rootless directed tree admits Wold-type decomposition. We begin with the following definition, whose motivation becomes clear from the proposition that follows.

\begin{definition}[Weak subsequence property]
    Let $T\in\mathcal B(\mathcal H)$ be a left-invertible operator and let $L$ be its Moore--Penrose inverse. We say that $T$ \textit{possesses the weak subsequence property} if there exists a total subset $\calF \subset \calH$ such that for every $x \in \calF$, $\{ T^n L^n x\}_{n \in \Z_{\ge 1}}$ has a weakly convergent subsequence.
\end{definition}

We remark that the weak subsequence property is equivalent to the existence of a total subset $\calF \subset \calH$ such that
\begin{equation*}
    \liminf_{n \to \infty} \norm{T^n L^n x} < \infty, \qquad x \in \calF.
\end{equation*}

The following proposition will help us to establish \Cref{thm:result-for-positive-weights}. A part of the proposition is essentially contained in the proof of \cite[Theorem 1]{MR936999} and is well known to experts.

    \begin{proposition}
    \label{prop:equivalences-of-WSP}
    Let $T\in\mathcal B(\mathcal H)$ be left-invertible, and let $L=(T^*T)^{-1}T^*$ be its Moore--Penrose inverse. If $T$ possesses the wandering subspace property, then it possesses the weak subsequence property. Moreover, if $T$ is analytic, then the converse holds. As a result, the following statements are equivalent whenever $T$ is analytic:
    \begin{enumerate}
        \item[\rm (i)] $T$ possesses the wandering subspace property.
        \item[\rm (ii)] $T$ possesses the weak subsequence property.
        \item[\rm (iii)] $T'$ is analytic.
    \end{enumerate}
    \end{proposition}
    
    \begin{proof}
    Define $ E:=\mathcal H\ominus T\mathcal H = \ker T^*$. We first show that if $T$ possesses the wandering subspace property then it possesses the weak subsequence property. To this end, define
    \begin{equation*}
        \calF := \operatorname{span} \{ T^j  E : j \in \Z_{\ge 0} \}.
    \end{equation*}
    It is easy to see that for each $x \in \calF$, there exists some $N \in \Z_{\ge 1}$ such that  $T^n L^n x = 0$ for all $n \in \Z_{\ge N}$. Moreover, since $T$ possesses the wandering subspace property, we have that $\bigvee \calF = \calH$. This shows that $T$ possesses the weak subsequence property.
    
    Conversely, assume that $T$ is analytic and possesses the weak subsequence property. Let $\mathcal F\subset\mathcal H$ be a total subset such that for every $x \in \calF$, $\{ T^n L^n x\}_{n \in \Z_{\ge 1}}$ has a weakly convergent subsequence. We show that $T$ possesses the wandering subspace property. Let $P:=I-TL$ be the orthogonal projection onto $ E$. Hence, for every $n\in \Z_{\ge 1}$,
    \begin{equation*}
    I-T^nL^n=\sum_{k=0}^{n-1}T^kPL^k,
    \end{equation*}
    and therefore $(I-T^nL^n)x\in [ E]_T$ for every $x\in\mathcal H$. This follows by induction on $n \in \Z_{\ge 1}$. See \cite[Equations (2.1) and (2.2)]{zbMATH01572594}.
    
    First of all, we fix $x\in\calF$ and choose a subsequence $(n_j)$ such that $T^{n_j}L^{n_j}x\to y$ weakly for some $y \in \calH$. Let $k\in\Z_{\ge1}$. For every $j\in\Z_{\ge k}$, we have $n_j\ge j\ge k$, and hence $T^{n_j}L^{n_j}x\in T^k\mathcal H$. Since $T^k\mathcal H$ is closed, and hence weakly closed, it follows that
    \begin{equation*}
    y\in\bigcap_{k\in \Z_{\ge 1}}T^k \mathcal H=\{0\},
    \end{equation*}
    where the last equality follows from the analyticity of $T$. Thus $T^{n_j}L^{n_j}x\to 0$ weakly, and consequently $(I-T^{n_j}L^{n_j})x\to x$ weakly.
    
    Since $[ E]_T$ is weakly closed, we have $x\in[ E]_T$. Hence, $\mathcal F\subseteq[ E]_T$, and it follows that $\mathcal H=[ E]_T$ as $\calF$ is a total set of $\calH$. Therefore, $T$ possesses the wandering subspace property.

    The equivalence of \rm{(i)} and \rm{(iii)} is precisely item \rm{(ii)} of \Cref{prop:analytic-wsp-in-terms-of-cauchy-dual}.
    \end{proof}

    \begin{corollary}
        Let $m \in \Z_{\ge 2}$. If $T$ is a norm-increasing $m$-concave operator possessing the wandering subspace property, then $T$ is analytic. Consequently, the following are equivalent for a norm-increasing $m$-concave operator:
        \begin{enumerate}
        \item[\rm (i)] $T$ possesses the wandering subspace property.
        \item[\rm (ii)] $T$ is analytic and possesses the weak subsequence property.
        \item[\rm (iii)] $T'$ is analytic.
    \end{enumerate}
    \end{corollary}
    \begin{proof}
        Let $T$ be a norm-increasing $m$-concave operator possessing the wandering subspace property. Then by \Cref{prop:Shimorin-hyper-range}, we have that $\calH_{\infty} (T)$ is reducing for $T$. It follows by \Cref{prop:hyperange-T-subset-hyperange-T'} that $\calH_{\infty} (T) \subset \calH_{\infty} (T')$. Since $T$ possesses the wandering subspace property, we have that $T'$ is analytic, that is $\calH_{\infty} (T') = \{ 0 \}$. Consequently, $\calH_{\infty} (T) = \{ 0\}$ and hence, $T$ is analytic. 

        The proof of the equivalences is immediate from \Cref{prop:equivalences-of-WSP}.
    \end{proof}

The proof of the following proposition is essentially the same as the argument used in \cite[Corollary 4.5]{MR5052237}. We provide an argument nonetheless for the sake of completeness.

\begin{proposition}
    Let $\scrT = ( V,  E)$ be a rootless directed tree, and let $S_{\lambda} \in \calB (\ell ^2 ( V))$ be a norm-increasing $3$-concave weighted shift with positive weights. Then $S_{\lambda}$ is either analytic or an isometry.
    \label{prop:analytic-or-isometry}
\end{proposition}

\begin{proof}
    If $S_\lambda$ were analytic, there would be nothing to prove. Hence, assume that $S_\lambda$ is not analytic. Since $S_\lambda$ is norm-increasing, it is left-invertible. We have that $\calH_\infty(S_\lambda)$ reduces $S_\lambda$ by \Cref{prop:Shimorin-hyper-range}. Consequently, $S_\lambda$ is balanced by \cite[Theorem 4.1]{MR5052237}.
    
    Since $S_\lambda$ is injective, $\scrT$ is leafless (see \cite[Proposition 3.1.7]{MR2919910}). Fix $v\in V$. Since $\scrT$ is rootless and leafless, we may choose a sequence $(v_n)_{n\in\Z}$ of vertices such that $v_0=v$ and $v_{n+1}\in\Chi(v_n)$ for every $n\in\Z$. Define $\gamma_n:=\norm{S_\lambda e_{v_n}}$ for each $n\in\Z$, and let $S_\gamma$ denote the bilateral weighted shift on $\ell^2(\Z)$ defined by $S_\gamma e_n=\gamma_ne_{n+1}$ for each $n \in \Z$.

    Since $S_\lambda$ is balanced, it can be shown using induction and \Cref{lem:JJS-powers} that $\norm{S_\lambda^k e_u}=\norm{S_\lambda^k e_w}$ for every $k\in\Z_{\geq0}$ whenever $u$ and $w$ belong to the same generation.

    Consequently, we have
    \begin{equation*}
        \norm{S_\lambda ^k e_{v_n}}^2 =\prod_{j=0}^{k-1}\gamma_{n+j}^2 =\norm{S_\gamma^ke_n}^2, \qquad k\in\Z_{\geq0},\ n\in\Z.
    \end{equation*}
      It follows by \Cref{lem:compute-beta-m} that $S_\gamma$ is norm-increasing and $3$-concave.

      Since $1\leq\gamma_n\leq\norm{S_\lambda}$ for every $n\in\Z$, the operator $S_\gamma$ is invertible. We have that $S_\gamma$ is unitary by using \Cref{prop:Shimorin-hyper-range} again. Thus $\gamma_n=1$ for every $n\in\Z$. Since $v\in V$ was arbitrary, $\norm{S_\lambda e_v}=1$ for every $v\in V$, and using \Cref{lem:compute-beta-m} with $m=1$, we have that $S_\lambda$ is an isometry.
\end{proof}

We now state a lemma which will be useful to prove \Cref{thm:result-for-positive-weights}. 

\begin{lemma}
    Let $\left\{ a_n \right\}_{n \in \Z_{\ge 1}}$ be a nonnegative nondecreasing sequence which is not identically zero. Define $\sigma_n := \frac{1}{n} \sum_{k=1}^{n} a_k$ for each $n \in \Z_{\ge 1}$. Then we have
    \begin{equation*}
        \liminf_{n \to \infty} \frac{a_n}{\sigma_n ^{2}} < \infty.
    \end{equation*}
    \label{lem:Cesaro}
\end{lemma}
\begin{proof}
    First, we observe that since $\{ a_n \}_{n \in \Z_{\ge1}}$ is nondecreasing, $\{ \sigma_{n} \}_{n \in \Z_{\ge 1}}$ is also nondecreasing. Indeed, by monotonicity of $\{ a_n \}_{n \in \Z_{\ge1}}$, we have $\sigma_n \le a_{n+1}$ for each $n \in \Z_{\ge 1}$, and consequently,
    \begin{equation*}
        \begin{aligned}
            \sigma_{n+1} - \sigma_n = \frac{1}{n+1} \sum_{k=1}^{n+1} a_k - \sigma_n 
                                    = \frac{n \sigma_n + a_{n+1}}{n+1} - \sigma_n 
                                    = \frac{a_{n+1} - \sigma_n}{n+1} \ge 0,
        \end{aligned}
    \end{equation*}
    for all $n \in \Z_{\ge 1}$.

    Moreover, since $\left\{ a_n \right\}_{n \in \Z_{\ge 1}}$ is not identically zero we have that there is some $N \in \Z_{\ge 1}$ such that $\sigma_n > 0$ for each $n \in \Z_{\ge N}$.

    Now, observe that $a_n = n \sigma_n - (n-1) \sigma_{n-1} = \sigma_n + (n-1) (\sigma_n - \sigma_{n-1})$ for each $n \in \Z_{\ge 2}$. Using the fact that $\sigma_{n-1} \le \sigma_{n}$ for each $n \in \Z_{\ge 2}$, it follows that for each $n \in \Z_{\ge N+1}$,
    \begin{equation}
        \begin{aligned}
        \frac{a_n}{\sigma_n ^2} &= \frac{\sigma_n + (n-1) (\sigma_n - \sigma_{n-1})}{\sigma_n ^2} \\
        &= \frac{1}{\sigma_n} + \frac{(n-1) (\sigma_n - \sigma_{n-1})}{\sigma_{n} ^2} \\
        &\le \frac{1}{\sigma_n} + (n-1) \left( \frac{1}{\sigma_{n-1}} - \frac{1}{\sigma_n} \right).
    \end{aligned}
    \label{eq:Cesaro-estimate}
    \end{equation}
    
    We define $c_n := \frac{1}{\sigma_{n-1}} - \frac{1}{\sigma_n}$ for each $n \in \Z_{\ge N+1}$. Observe that $c_n \ge 0$ for each $n \in \Z_{\ge N+1}$ and $\sum_{n = N+1}^{\infty} c_n < \infty$. Consequently, $\liminf_{n \to \infty} nc_n = 0$. Indeed, otherwise there would exist some $\varepsilon > 0$ and some $N_0 \in \Z_{\ge N+1}$ such that $c_n \ge \varepsilon/n$ for all $n \in \Z_{\ge N_0}$, but this would imply that $\sum_{n = N+1}^{\infty} c_n = \infty$.

    We now select a strictly increasing sequence $\{ n_j \}_{j \in \Z_{\ge 1}} \subset \Z_{\ge N+1}$ such that $n_j c_{n_j} \to 0$. From \eqref{eq:Cesaro-estimate}, we have that 
    \begin{equation*}
        \frac{a_{n_j}}{\sigma_{n_j} ^2} \le \frac{1}{\sigma_N} + (n_j - 1) c_{n_j}, \qquad j \in \Z_{\ge 1}.
    \end{equation*}
    Since the right hand side of the previous inequality is bounded, we have $\liminf_{n \to \infty} \frac{a_{n}}{\sigma_{n} ^2} < \infty$. This completes the proof of the lemma.
\end{proof}

\begin{theorem}
    Let $\scrT = ( V,  E)$ be a rootless directed tree, and let $S_{\lambda} \in \calB (\ell ^2 ( V))$ be an analytic, norm-increasing $3$-concave weighted shift with positive weights. Then $S_{\lambda}$ possesses the wandering subspace property.
    \label{thm:result-for-positive-weights}
\end{theorem}

\begin{proof}
    We begin by defining $\beta_j (v) = \ip{\beta_j (S_{\lambda})e_v, e_v}$ for each $j \in \{1,2\}$ and each $v\in  V$. Since $S_{\lambda}$ is norm-increasing, we have $\beta_1 (S_{\lambda}) \ge 0$. We have by using \Cref{lem:m-concave} and $3$-concavity that $\beta_2 (S_{\lambda}) \ge 0$ and 
    \begin{equation}
        S_{\lambda}^{*n}S_{\lambda}^{n} \le I + n \beta_1 (S_{\lambda}) + \binom{n}{2} \beta_2 (S_{\lambda}), \qquad n \in \Z_{\ge 0}.
        \label{eq:shift-3-concave}
    \end{equation}
    As $\beta_j (S_\lambda) \ge 0$ for each $j \in \{1,2\}$ and $S_{\lambda}$ is bounded, we choose $M > 0$ such that $0 \le \beta_j (v) \le M$ for every $v \in  V$ and every $j \in \{1,2\}$.

    Now, we fix $u \in  V$ and define $u_n := \pr^{\ip{n}} (u)$ for each $n \in \Z_{\ge 0}$. In view of \Cref{prop:equivalences-of-WSP}, we show that $\{ S_{\lambda} ^{n} L_{\lambda} ^{n} e_u \}_{n \in \Z_{\ge 0}}$ has a bounded subsequence. To this end, we define for $n \in \Z_{\ge 0}$,
    \begin{equation*} 
    D_n := \frac{1}{\lambda_{u_n\mid u} ^{2}} \prod_{k=1}^{n}\bigl(1+\beta_1(u_k)\bigr)\quad \text{and} \quad
    H_n := \sum_{k=1}^{n} \frac{\beta_2(u_k)}{\lambda_{u_k\mid u} ^{2}}.
    \end{equation*}
    Using \Cref{lem:all-about-L-lambda} and equation \eqref{eq:shift-3-concave} to estimate $\norm{S_{\lambda} ^{n} L_{\lambda} ^{n} e_u}$ for $n \in \Z_{\ge 0}$, we have
    \begin{equation} 
           \norm{S_{\lambda} ^{n} L_{\lambda}^{n} e_u} ^2 \le \mathbf{I}_n + \mathbf{II}_n + \mathbf{III}_n, \qquad n \in \Z_{\ge 0},
           \label{eq:summands-of-S-L}
    \end{equation}
    where we define for each $n \in \Z_{\ge 0}$,
    \begin{equation*}
        \mathbf{I}_n  := \frac{1}{\lambda_{u_n \mid u} ^{2} D_{n}^{2}}, \quad \mathbf{II}_n := \frac{n \beta_1 (u_n)}{\lambda_{u_n \mid u} ^{2} D_{n}^{2}}, \quad \mathbf{III}_n := \frac{\binom{n}{2} \beta_2 (u_n)}{\lambda_{u_n \mid u} ^{2} D_{n}^{2}}.
    \end{equation*}

    In view of inequality \eqref{eq:summands-of-S-L}, it suffices for us to show that $\{ \mathbf{I}_n \}_{n\in\Z_{\ge 0}}$ and $\{ \mathbf{II}_n \}_{n\in\Z_{\ge 0}}$  are bounded and that $\{ \mathbf{III}_n \}_{n\in\Z_{\ge 0}}$  has a bounded subsequence. We make the following claims in order to establish these assertions.

    \textit{Claim 1.} The following inequalities hold for all $n \in \Z_{\ge 0}$,
    \begin{equation*}
        \begin{gathered}
            D_n \ge 1, \qquad 
            D_n \ge \frac{1}{\lambda_{u_n \mid u} ^{2}} + \sum_{k=1}^{n} \frac{\beta_1 (u_k)}{\lambda_{u_k \mid u} ^{2}}, \qquad H_n \le MD_n.
        \end{gathered}
    \end{equation*}
    Moreover, the sequence $\left\{ \frac{\beta_2 (u_n)}{\lambda_{u_n \mid u} ^2} \right\}_{n \in \Z_{\ge 0}}$ is nonnegative and nondecreasing. If the previous sequence is identically zero, then the sequence $\left\{ \frac{\beta_1 (u_n)}{\lambda_{u_n \mid u} ^2} \right\}_{n \in \Z_{\ge 0}}$ is nonnegative and nondecreasing.

    \begin{proof}[Proof of Claim 1.] The following equalities $S_{\lambda}^* S_{\lambda} =I+\beta_{1}(S_\lambda)$,  $S_{\lambda}^*\beta_1(S_{\lambda})S_{\lambda} = \beta_{1}(S_{\lambda})+\beta_{2}(S_{\lambda})$ and the inequality $S_{\lambda}^*\beta_{2}(S_{\lambda})S_{\lambda} = \beta_{2}(S_{\lambda})+\beta_{3}(S_{\lambda}) \le \beta_{2}(S_{\lambda})$
    together with \Cref{lem:JJS-powers,lem:compute-beta-m} imply that the following holds for each $v \in  V$,
    \begin{equation}
        \begin{aligned}
         1 +\beta_1 (v) &= \sum_{w \in \Chi (v)} \lambda_w^{2}, \\
        \beta_1(v) + \beta_2 (v) &= \sum_{w \in \Chi (v)} \lambda_w ^{2} \beta_1 (w),  \\
        \beta_2 (v) & \ge \sum_{w\in \Chi (v) } \lambda_w ^{2} \beta_2 (w).
        \end{aligned}
        \label{eq:some-identities}
    \end{equation}
   Since $u_{k-1} \in \Chi (u_k)$ for each $k \in \Z_{\ge 1}$, we have $\lambda_{u_k \mid u} = \lambda_{u_{k-1}} \lambda_{u_{k-1} \mid u}$ for each $k \in \Z_{\ge 1}$. Consequently, from \eqref{eq:some-identities}, we have for $k \in \Z_{\ge 1}$,
    \begin{equation}
        \left( \frac{1}{\lambda_{u_k \mid u}^{2}} - \frac{1}{\lambda_{u_{k-1} \mid u} ^{2}} \right) + \frac{\beta_1 (u_k)}{\lambda_{u_k \mid u}^{2}} = \frac{1}{\lambda_{u_k \mid u} ^{2}} \sum_{\substack{w \in \Chi (u_k) \\ w \ne u_{k-1}}} \lambda_w ^{2},
        \label{eq:claim-1-first-identity}
    \end{equation}
    \begin{equation}
        \left( \frac{\beta_1(u_k)} {\lambda_{u_k \mid u}^2} - \frac{\beta_1(u_{k-1})} {\lambda_{u_{k-1}\mid u}^2} \right) + \frac{\beta_2(u_k)} {\lambda_{u_k \mid u}^2} = \frac{1}{\lambda_{u_k\mid u}^2} \sum_{\substack{w \in \Chi (u_k) \\ w \ne u_{k-1}}} \lambda_w^2\beta_1(w),
        \label{eq:claim-1-second-identity}
    \end{equation}
    and, 
    \begin{equation}
        \frac{\beta_2(u_k)}{\lambda_{u_k\mid u}^2} \ge \frac{\beta_2(u_{k-1})} {\lambda_{u_{k-1}\mid u}^2} + \frac{1}{\lambda_{u_k\mid u}^2} \sum_{\substack{w \in \Chi (u_k) \\ w \ne u_{k-1}}} \lambda_w^2\beta_2(w).
        \label{eq:claim-1-third-identity}
    \end{equation}
    Now, observe that inequality \eqref{eq:claim-1-third-identity} implies that the sequence $\left \{ \frac{\beta_2 (u_n)}{\lambda_{u_n \mid u}^2} \right\} _{n \in \Z_{\ge 0}}$ is nonnegative and nondecreasing. If the aforementioned sequence is identically zero, then equation \eqref{eq:claim-1-second-identity} implies $\left \{ \frac{\beta_1 (u_n)}{\lambda_{u_n \mid u}^2} \right\} _{n \in \Z_{\ge 0}}$ is nonnegative and nondecreasing.

    We multiply both sides of \eqref{eq:claim-1-first-identity} by $\prod_{j=1}^{k-1}(1+\beta_1(u_j))$ to obtain, for each $k\in\Z_{\ge1}$,
    \begin{equation*}
        D_k-D_{k-1}
        =\frac{1}{\lambda_{u_k\mid u}^2}
        \left(\prod_{j=1}^{k-1}(1+\beta_1(u_j))\right)
        \sum_{\substack{w\in\Chi(u_k)\\w\ne u_{k-1}}}\lambda_w^2.
    \end{equation*}
    Since $\beta_1 (u_k) \ge 0$ for each $k \in \Z_{\ge 0}$, summing over $k=1,2, \ldots, n$ gives
    \begin{equation}
        D_n \ge 1 + \sum_{k=1}^n  \frac{1}{\lambda_{u_k \mid u} ^{2}} \sum_{\substack{w \in \Chi (u_k) \\ w \ne u_{k-1}}} \lambda_{w} ^{2}.
        \label{eq:D_n-lowerbound}
    \end{equation}
    This shows that $D_n \ge 1$ for each $n \in \Z_{\ge 0}$.

    Again, taking sums in \eqref{eq:claim-1-first-identity}, we obtain for each $n \in \Z_{\ge 0}$,
    \begin{equation*}
        \frac{1}{\lambda_{u_n\mid u}^2} + \sum_{k=1}^n \frac{\beta_1(u_k)}{\lambda_{u_k\mid u}^2} = 1+ \sum_{k=1}^n \frac{1}{\lambda_{u_k\mid u}^2} \sum_{\substack{w \in \Chi (u_k) \\ w \ne u_{k-1}}}\lambda_w^2.
    \end{equation*}
    From the previous equality and inequality \eqref{eq:D_n-lowerbound}, we have for each $n \in \Z_{\ge 0}$ that
    \begin{equation*}
        D_n \ge \frac{1}{\lambda_{u_n \mid u} ^{2}} + \sum_{k=1}^{n} \frac{\beta_1 (u_k)}{\lambda_{u_k \mid u} ^{2}}.
    \end{equation*}

    Taking sums in \eqref{eq:claim-1-second-identity}, we have for each $n \in \Z_{\ge 0}$,
    \begin{equation*}
        H_n = \beta_1(u) - \frac{\beta_1(u_n)}{\lambda_{u_n\mid u}^2}+ \sum_{k=1}^n \frac{1}{\lambda_{u_k\mid u}^2} \sum_{\substack{w \in \Chi(u_k) \\ w\ne u_{k-1}}} \lambda_w^2\beta_1(w).
    \end{equation*}

    Since $0 \le \beta_1 (v) \le M$, we have using \eqref{eq:D_n-lowerbound} that 
    \begin{equation*}
        H_n \le M \left( 1+ \sum_{k=1}^n \frac{1}{\lambda_{u_k\mid u}^2} \sum_{\substack{w \in \Chi(u_k) \\ w\ne u_{k-1}}}\lambda_w^2 \right) \le MD_n.
    \end{equation*}
    This completes the proof of Claim 1.
\end{proof}

    \textit{Claim 2.} The sequences $\left\{ \mathbf{I}_n \right\}_{n \in \Z_{\ge 0}}$ and $\left\{ \mathbf{II}_n \right\}_{n \in \Z_{\ge 0}}$ are bounded. 
    \begin{proof}[Proof of Claim 2.] Note that from Claim 1, we have $D_n \ge \frac{1}{\lambda_{u_n \mid u} ^{2}}$ and $D_n \ge 1$ for each $n \in \Z_{\ge 0}$. Consequently, we have
    \begin{equation*}
        \mathbf{I}_n = \frac{1}{\lambda_{u_n \mid u} ^{2} D_n ^{2}} \le \frac{1}{D_n} \le 1, \qquad n \in \Z_{\ge 0}.
    \end{equation*}
    This shows that the sequence $\left\{ \mathbf{I}_n \right\}_{n \in \Z_{\ge 0}}$ is bounded. We proceed to show that the sequence $\left\{ \mathbf{II}_n \right\}_{n \in \Z_{\ge 0}}$ is bounded.
    
    Since $\beta_1 (u_n) \le M$ for each $n \in \Z_{\ge 0}$, we again have by Claim 1 that 
    \begin{equation}
        \mathbf{II}_n = \frac{n \beta_1 (u_n)}{\lambda_{u_n \mid u}^2 D_n ^{2}} \le \frac{M n}{D_n}, \qquad n \in \Z_{\ge 0}.
        \label{eq:estimate-II_n}
    \end{equation}

    Suppose now that the sequence $\left\{ \frac{\beta_2 (u_n)}{\lambda_{u_n \mid u} ^2} \right\}_{n \in \Z_{\ge 0}}$ is not identically zero. By Claim 1, $\left\{ \frac{\beta_2 (u_n)}{\lambda_{u_n \mid u} ^2} \right\}_{n \in \Z_{\ge 0}}$ is nonnegative and nondecreasing. It follows that there is some $N \in \Z_{\ge 1}$ and some $\delta > 0$ such that $\frac{\beta_2 (u_k)}{\lambda_{u_k \mid u} ^2} \ge \delta$ for all $k \in \Z_{\ge N}$. Consequently, we have
    \begin{equation*}
        H_n = \sum_{k=1}^{n} \frac{\beta_2 (u_k)}{\lambda_{u_k \mid u} ^2} \ge (n-N+1)\delta, \qquad n \in \Z_{\ge N}.
    \end{equation*}
    Since $H_n \le M D_n$ for each $n \in \Z_{\ge 0}$ by Claim 1, we have
    \begin{equation*}
        D_n \ge \frac{\delta}{M} (n-N+1) \ge \frac{\delta}{2M}n, \qquad n \in \Z_{\ge 2N}.
    \end{equation*}
    It follows from \eqref{eq:estimate-II_n} that $\mathbf{II}_n \le \frac{2M^2}{\delta}$ for each $n \in \Z_{\ge 2N}$. Therefore, in the case where the sequence $\left\{ \frac{\beta_2 (u_n)}{\lambda_{u_n \mid u} ^2} \right\}_{n \in \Z_{\ge 0}}$ is not identically zero, the sequence $\left\{ \mathbf{II}_n \right\}_{n \in \Z_{\ge 0}}$ is bounded.

    We consider the case now where the sequence $\left\{ \frac{\beta_2 (u_n)}{\lambda_{u_n \mid u} ^2} \right\}_{n \in \Z_{\ge 0}}$ is identically zero. Then by Claim 1, $\left\{ \frac{\beta_1 (u_n)}{\lambda_{u_n \mid u} ^2} \right\}_{n \in \Z_{\ge 0}}$ is nonnegative and nondecreasing. If this sequence is identically zero, then $\mathbf{II}_n = 0$ for each $n \in \Z_{\ge 0}$ and we are done. Hence, we may assume without loss of generality that this sequence is not identically zero.

    Consequently,  there is some $N \in \Z_{\ge 1}$ and some $\delta> 0$ such that $\frac{\beta_1 (u_k)}{\lambda_{u_k \mid u}^2} \ge \delta$ for each $k \in \Z_{\ge N}$. It then follows from Claim 1 that 
    \begin{equation*}
        D_n \ge \sum_{k=1}^{n} \frac{\beta_1 (u_k)}{\lambda_{u_k \mid u}^2} \ge (n -N +1)\delta \ge \frac{\delta}{2} n, \qquad n \in \Z_{\ge 2N}.
    \end{equation*}
    Hence, we have from \eqref{eq:estimate-II_n} that $\mathbf{II}_n \le \frac{2M}{\delta}$ for each $n \in \Z_{\ge 2N}$. Thus,  the sequence $\left\{ \mathbf{II}_n \right\}_{n \in \Z_{\ge 0}}$ is bounded in this case as well.

    This completes the proof of Claim 2.
\end{proof}

\textit{Claim 3.} The sequence $\{ \mathbf{III}_n \}$ has a bounded subsequence.
\begin{proof}[Proof of Claim 3.]
    We have by Claim 1 that $\left\{ \frac{\beta_2 (u_n)}{\lambda_{u_n \mid u}^2} \right\}_{n \in \Z_{\ge 0}}$ is nonnegative and nondecreasing. If this sequence is identically zero, then $\mathbf{III}_n = 0$ for each $n \in \Z_{\ge 0}$ and there is nothing to prove. So, we may assume $\left\{ \frac{\beta_2 (u_n)}{\lambda_{u_n \mid u}^2} \right\}_{n \in \Z_{\ge 0}}$ is not identically zero.

    Recall that $H_n = \sum_{k=1}^{n} \frac{\beta_2 (u_k)}{\lambda_{u_k \mid u}^2}$ for each $n \in \Z_{\ge 0}$. It follows from \Cref{lem:Cesaro} that there is a strictly increasing sequence $\{n_j\}_{j\in\Z_{\ge1}} \subset \Z_{\ge 1}$ and some $C > 0$ such that
    \begin{equation*}
        \frac{n_j ^2 \beta_{2} (u_{n_j})}{\lambda_{u_{n_j} \mid u }^2 H_{n_j}^2} \le C, \qquad j \in \Z_{\ge 1}.
    \end{equation*}

    By Claim 1, we also have that $H_n \le M D_n$ for each $n \in \Z_{\ge 0}$ and consequently, for each $j \in \Z_{\ge 1}$,
    \begin{align*}
        \mathbf{III}_{n_j} &= \binom{n_j}{2}\frac{\beta_2(u_{n_j})}{\lambda_{u_{n_j}\mid u}^2D_{n_j}^2} \leq M^2\binom{n_j}{2}\frac{\beta_2(u_{n_j})}{\lambda_{u_{n_j}\mid u}^2H_{n_j}^2} \\
        &\qquad \le \frac{M^2}{2}\frac{n_j^2\beta_2(u_{n_j})}{\lambda_{u_{n_j}\mid u}^2H_{n_j}^2} \leq \frac{M^2C}{2}.
    \end{align*}
    Hence, $\left\{ \mathbf{III}_{n_j} \right\}$ is bounded. This completes the proof of Claim 3.
\end{proof}

We can now finish the proof. It follows by inequality \eqref{eq:summands-of-S-L} and Claims 2 and 3 that $\left\{ S_\lambda ^n L_{\lambda} ^{n} e_u \right\}_{n \in \Z_{\ge 0}}$ has a bounded subsequence. Since $\{ e_u : u \in  V \}$ is a total set for $\ell ^2 ( V)$, the proof is complete by \Cref{prop:equivalences-of-WSP}.
\end{proof}

\begin{proposition}
    Let $\scrT = ( V,  E)$ be a directed tree, and let $S_{\lambda} \in \calB (\ell ^2 ( V))$ be a norm-increasing $3$-concave weighted shift with positive weights. Then $S_{\lambda}$ admits Wold-type decomposition.
    \label{prop:wold-type}
\end{proposition}

\begin{proof}
    If $\scrT = ( V,  E)$ is rooted, we are done in view of \cite[Proposition 1.3.4]{MR3740250}. So we assume without any loss of generality that $\scrT$ is rootless. 
    
    If $S_{\lambda}$ is an isometry, then there is nothing to prove, in view of the classical Wold decomposition; see \cite[Theorem 1.1]{MR2760647}. So, we may assume that $S_{\lambda}$ is not an isometry. By \Cref{prop:analytic-or-isometry}, we have that $S_{\lambda}$ must be analytic. It follows that $S_{\lambda}$ possesses the wandering subspace property by \Cref{thm:result-for-positive-weights}. Hence, $S_{\lambda}$ admits Wold-type decomposition. This completes the proof. 
\end{proof}

The following proposition shows that every weighted shift on a directed tree can be decomposed as an orthogonal direct sum of weighted shifts on directed trees with nonzero weights. This can be viewed as a generalisation of \cite[Proposition 3.1.6]{MR2919910} and can be deduced from \cite[Proposition 3.2 and Lemma 3.5]{zbMATH07570541}. We omit the proof, noting that the result follows immediately and can also be obtained without appealing to the aforementioned reference. Recall that if $W \subset V$, we treat $\ell^2(W)$ as a closed linear subspace of $\ell^2(V)$ by identifying each $f\in\ell^2(W)$ with its extension to $V$ which vanishes on $V\setminus W$.

\begin{proposition}
    \label{prop:tree-decomposition}
    Let $\scrT = ( V,  E)$ be a directed tree, and let $S_{\lambda} \in \calB (\ell ^2 (V))$ be a weighted shift with weights $\lambda = \{ \lambda_v \}_{v \in V^{\circ}}$. Define $Z_{\lambda} := \{ v \in  V^{\circ} \, : \, \lambda_{v} = 0 \}$ and $ E_{\lambda} :=  E \setminus \{ (\pr (v), v) : v \in Z_{\lambda} \}$. Let $\{  V_{\alpha} \}_{\alpha \in I}$ be the vertex sets of the connected components of the underlying undirected graph of $( V,  E_{\lambda})$. For each $\alpha \in I$, we define the graph $\scrT_{\alpha} = ( V_{\alpha},  E_{\alpha})$ whose edge set is defined by $ E_{\alpha} :=  E_{\lambda} \cap ( V_{\alpha} \times V_{\alpha})$. Define the system of weights $\lambda ^{(\alpha )} = \{ \lambda ^{(\alpha )}_{v} \}_{v \in  V_{\alpha} ^{\circ}}$ on $\scrT_{\alpha}$ for each $\alpha \in I$ by
    \begin{equation*}
        \lambda_{v}^{(\alpha)} := \lambda_{v}, \qquad \alpha \in I, \; v \in  V_{\alpha}^{\circ}.
    \end{equation*}
    Then the following statements hold:
    \begin{enumerate}
        \item[(i)] $\scrT_{\alpha}$ is a directed tree for each $\alpha \in I$,
        \item[(ii)]  $\lambda_{v}^{(\alpha)} \ne 0$ for each $\alpha \in I$ and every $v \in  V_{\alpha}^{\circ}$,
        \item[(iii)] The subspace $\ell ^2 ( V_{\alpha})$ reduces $S_{\lambda}$ and $S_{\lambda} \mid _{\ell ^2 ( V _{\alpha})} =  S_{\lambda ^{ (\alpha)}}$.
        \item [(iv)] $\ell ^2 ( V ) = \bigoplus _{\alpha \in I} \ell ^2 (V_{\alpha})$ and $S_{\lambda} = \bigoplus_{\alpha \in I} S_{\lambda^{(\alpha)}}$.
    \end{enumerate}
\end{proposition}

We state an operator-theoretic lemma before we proceed to prove \Cref{thm:3-concave}.

\begin{lemma}
\label{lem:direct-sum}
Let $\{ \calH _j \}_{j\in J}$ be a family of Hilbert spaces, let $T_j\in\mathcal{B}(\calH _j)$ with $\sup_{j\in J}\|T_j\|<\infty$, and define $\calH=\bigoplus_{j\in J}\calH _j$ and $T=\bigoplus_{j\in J} T_j$. Then the following statements hold:
\begin{enumerate}
    \item[(i)] $T$ is norm-increasing if and only if $T_j$ is norm-increasing for each $j \in J$.
    \item[(ii)] For every $m\in\Z_{\ge 2}$, $T$ is $m$-concave if and only if $T_j$ is $m$-concave for each $j \in J$.
    \item[(iii)]  The following holds: $[\ker T^*]_T = \bigoplus_{j\in J} \, [\ker T_j^*]_{T_j}.$
    \item[(iv)] If $T$ is norm-increasing, then $\calH_{\infty} (T) = \bigoplus_{j\in J} \calH_{\infty} (T_j).$
\end{enumerate}

\end{lemma}

Before we prove \Cref{thm:3-concave}, we restate it again here for convenience.

    \threeConcave*
    \begin{proof}
        In view of \cite[Theorem 3.2.1]{MR2919910}, we may assume without any loss of generality that the weights are nonnegative. By using \Cref{prop:tree-decomposition}, we have that there exist directed trees $\{\scrT_\alpha=( V_\alpha, E_\alpha)\}_{\alpha\in A}$ and weighted shifts $S_{\lambda^{(\alpha)}}$ with positive weights such that $S_\lambda = \bigoplus_{\alpha\in A}S_{\lambda^{(\alpha)}}$.
        Since each $\ell^2( V_\alpha)$ reduces $S_\lambda$, every $S_{\lambda^{(\alpha)}}$ is norm-increasing and $3$-concave by items $(i)$ and $(ii)$ of \Cref{lem:direct-sum}. Hence, each $S_{\lambda^{(\alpha)}}$ has the Wold-type decomposition by \Cref{prop:wold-type}. The rest of the proof now follows from items $(iii)$ and $(iv)$ of \Cref{lem:direct-sum}.
    \end{proof}

\section{When does a generalised root exist?}\label{sec:generalised-root}

In this section, we investigate when directed trees admit a generalised root. We recall the definition of a generalised root:
\begin{definition}\label{def:generalised-root}
    Let $\scrT=( V, E)$ be a rootless directed tree. A vertex $\omega\in V_{\prec}$ is called a \textit{generalised root of $\scrT$} if
    \begin{equation}
        \label{eq:generalised-root-1}
        \abs{\Chi (\pr ^{\ip{k}} (\omega))} =1, \qquad k \in \Z_{\ge 1}.
    \end{equation}
\end{definition}

\begin{remark}
    A rootless directed tree admits at most one generalised root.
\end{remark}

\begin{proposition}
    \label{prop:generalised-root-not-analytic}
    Let $\scrT=(V,E)$ be a rootless directed tree which admits a generalised root, and let $S_{\lambda}\in\calB(\ell^2(V))$ be a weighted shift with positive weights. Then $S_{\lambda}$ is not analytic.
\end{proposition}

\begin{proof}
    Let $\omega$ be the generalised root of $\scrT$, and define $u_n:=\pr^{\ip{n}}(\omega)$ for each $n\in\Z_{\ge 0}$. Since $\omega$ is a generalised root, we have $\abs{\Chi(u_n)}=1$ for every $n\in\Z_{\ge 1}$. Since $u_{n-1}\in\Chi(u_n)$, it follows that $\Chi(u_n)=\{u_{n-1}\}$ for every $n\in\Z_{\ge 1}$. Hence, $\Chi^{\ip{n}}(u_n)=\{\omega\}$ for every $n\in\Z_{\ge 1}$. By \Cref{lem:JJS-powers}, we have that $S_{\lambda}^{n}e_{u_n}=\lambda_{u_n\mid\omega}e_{\omega}$ for every $n\in\Z_{\ge 1}$. Since the weights are positive, $\lambda_{u_n\mid\omega}>0$. Thus, $e_{\omega}\in S_{\lambda}^{n}(\ell^2(V))$ for every $n\in\Z_{\ge 1}$. Consequently, $e_{\omega}\in\calH_{\infty}(S_{\lambda})$, and hence $S_{\lambda}$ is not analytic.
\end{proof}

\begin{definition}[{\cite[Definitions 1 and 3]{MR3532172}}]
    Let $\scrT=( V, E)$ be a directed tree. The \emph{branching index} of $\scrT$, denoted by $b_{\scrT}$, is defined as follows:

    If $\scrT$ is rooted, then
    \begin{equation*}
        b_{\scrT} := \begin{cases}
            1+\sup\{n_w:w\in V_{\prec}\}, &  V_{\prec}\ne\varnothing,\\
            0, &  V_{\prec}=\varnothing,
        \end{cases}
    \end{equation*}
    where $n_w \in \Z_{\ge 0}$ is the unique nonnegative integer satisfying $w \in \Chi ^{\ip{n_w}} (\mathsf{root})$ for $w \in  V$.

    If $\scrT$ is rootless, then
    \begin{equation*}
        b_{\scrT} := \inf \left\{ m\in\Z_{\ge 0}: \Chi^{\ip{k}}( V_{\prec})\cap V_{\prec} =\varnothing \text{ for every } k\ge m \right\}.
    \end{equation*}
    In either case, we say $\scrT$ has \textit{finite branching index} if $b_{\scrT} < \infty$; else we say $\scrT$ has \textit{infinite branching index}.
\end{definition}

The following proposition, whose proof we omit, is a straightforward extension of \cite[Lemma 6.1]{MR3532172}. In particular, it shows that a rootless directed tree with finite branching index admits a generalised root if and only if it has a branching vertex, in which case the generalised root is unique.

\begin{proposition}
\label{prop:characterisation-of-finite-branching-rootless-trees}
    Let $\scrT=( V, E)$ be a rootless directed tree having finite branching index. Then exactly one of the following holds:
    \begin{enumerate}
        \item[\rm (i)] $ V_{\prec}=\varnothing$ and $\scrT$ is leafless, in which case $\scrT$ is isomorphic to the directed tree on $\Z$ given by $\Chi(n)=\{n+1\}$ for each $n\in\Z$;
        \item[\rm (ii)] $ V_{\prec}=\varnothing$ and $\scrT$ is not leafless, in which case $\scrT$ is isomorphic to the directed tree on $\Z_{\ge 0}$ given by $\Chi(n+1)=\{n\}$ for each $n\in\Z_{\ge 0}$;
        \item[\rm (iii)] $ V_{\prec}\ne\varnothing$, in which case $\scrT$ admits a unique generalised root.
    \end{enumerate}
\end{proposition}

If $\scrT$ is a rootless directed tree with finite branching index, $ V_{\prec}\ne\emptyset$, and $\omega$ is its generalised root, then the directed subtree $\scrT_{\omega}$ induced by $\mathsf{Des}(\omega)$ is a rooted directed tree with root $\omega$ and finite branching index.
Moreover, $b_{\scrT}=b_{\scrT_{\omega}}$.
Thus, $\scrT$ may be viewed as the graph obtained by attaching
$\scrT_{\omega}$ to the one-sided infinite directed path
\begin{equation*}
    \cdots \longrightarrow \pr^{\ip{3}}(\omega)
    \longrightarrow \pr^{\ip{2}}(\omega)
    \longrightarrow \pr(\omega)
\end{equation*}
by the edge $(\pr(\omega),\omega)$. See \Cref{fig:finite-branching-index} for an illustration of a rootless directed tree with finite branching index which has a branching vertex.

    \begin{figure}[H]
        \centering  
        \includegraphics[width=0.4\textwidth]{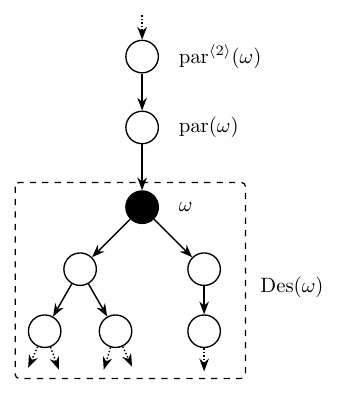}
        \caption{A rootless directed tree with finite branching index}
        \label{fig:finite-branching-index}
    \end{figure}

In the next proposition, we show that if a rootless directed tree has a branching vertex and supports a weighted shift $S_\lambda$ such that $\dim\ker S_\lambda^*<\infty$, then the tree admits a unique generalised root.

\begin{proposition}
\label{prop:finite-codimension}
    Let $\scrT=( V, E)$ be a directed tree, and let $S_\lambda\in\calB(\ell^2( V))$ be a bounded weighted shift with complex weights. If $\dim\ker S_\lambda^*<\infty$, then $\scrT$ is locally finite and has only finitely many branching vertices. In particular, $b_{\scrT}<\infty$.

    Moreover, if $\scrT$ is rootless and $ V_{\prec} \ne \emptyset$, then $\scrT$ admits a unique generalised root.
\end{proposition}

\begin{proof}
    By \cite[Proposition 3.5.1]{MR2919910}, we have that
    \begin{equation*}
        \ker S_{\lambda}^{*}=\ell^2(\mathsf{Root}(\scrT))\oplus\bigoplus_{u\in V'}\left(\ell^2(\Chi(u))\ominus\ip{\lambda^{u}}\right),
    \end{equation*}
    where $\lambda^{u}:\Chi(u)\to\C$ is the function in $\ell^2(\Chi(u))$ given by $\lambda^{u}(v)=\lambda_v$ for each $v\in\Chi(u)$. Here $\langle\lambda^u\rangle$ denotes the linear span of $\{\lambda^u\}$. Since $\ker S_{\lambda}^{*}$ is finite dimensional, we have that $\ell^2(\Chi(u))\ominus\ip{\lambda^{u}}$ is finite dimensional for each $u\in V'$. Since $\ip{\lambda^{u}}$ has dimension at most one, it follows that $\ell^2(\Chi(u))$ is finite dimensional for each $u\in V'$. Hence, $\Chi(u)$ is finite for each $u\in V'$. If $u\in V\setminus V'$, then $\Chi(u)=\varnothing$. Therefore, $\scrT$ is locally finite.

    Define $F:=\{u\in V':\ell^2(\Chi(u))\ominus\ip{\lambda^{u}}\ne\{0\}\}$. Since $\ker S_{\lambda}^{*}$ is finite dimensional, the set $F$ must be finite. If $u\in V_{\prec}$, then $\abs{\Chi(u)}\ge2$, and since $\dim\ip{\lambda^{u}}\le1$, we have $\ell^2(\Chi(u))\ominus\ip{\lambda^{u}}\ne\{0\}$. Thus, $ V_{\prec}\subseteq F$, and consequently, $ V_{\prec}$ is finite. It follows from the definition of the branching index that $b_{\scrT}<\infty$.

    Suppose, moreover, that $\scrT$ is rootless and $ V_{\prec} \ne \emptyset$. Since $b_{\scrT}<\infty$, it follows from \Cref{prop:characterisation-of-finite-branching-rootless-trees} that $\scrT$ admits a unique generalised root.
\end{proof}

    \begin{proposition}
        \label{prop:some-finite-generation-implies-generalised-root}
        Let $\scrT=( V, E)$ be a leafless rootless directed tree such that $ V_{\prec}\ne\varnothing$. If $\scrG_u$ is finite for some $u\in V$, then $\scrT$ admits a unique generalised root.
    \end{proposition}
    
    \begin{proof}
        Let $u\in V$ such that $\scrG_u$ is finite, and define $u_n:=\pr^{\ip{n}}(u)$ for each $n\in\Z_{\ge 0}$. By \cite[Proposition 2.1.12(viii)]{MR2919910}, we have that $\Chi(\scrG_{u_{n+1}})=\scrG_{u_n}$ for each $n\in\Z_{\ge 0}$. Since $\scrT$ is leafless, the map $\pr:\scrG_{u_n}\to\scrG_{u_{n+1}}$ is onto. Hence, $\abs{\scrG_{u_{n+1}}}\le\abs{\scrG_{u_n}}$ for each $n\in\Z_{\ge 0}$. Thus, $\{\abs{\scrG_{u_n}}\}_{n\in\Z_{\ge 0}}$ is a non-increasing sequence of positive integers, and hence there exists $N\in\Z_{\ge 0}$ such that $\abs{\scrG_{u_{n+1}}}=\abs{\scrG_{u_n}}$ for every $n\ge N$. Therefore, $\pr:\scrG_{u_n}\to\scrG_{u_{n+1}}$ is bijective for every $n\ge N$. Since $\Chi(\scrG_{u_{n+1}})=\scrG_{u_n}$, it follows that $\abs{\Chi(v)}=1$ for every $v\in\scrG_{u_{n+1}}$ and every $n\in \Z_{\ge N}$. In particular, $\scrG_{u_n}\cap V_{\prec}=\varnothing$ for every $n\in \Z_{\ge N+1}$.
    
        Fix $w\in V_{\prec}$. By \cite[Proposition 2.1.4]{MR2919910}, there is some $v\in V$ such that $\{u,w\}\subset\mathsf{Des}(v)$. Hence, we have that there are some $r,s\in\Z_{\ge 0}$ such that $v=\pr^{\ip{r}}(w)=\pr^{\ip{s}}(u)$. Thus, $\pr^{\ip{r+k}}(w)=u_{s+k}$ for every $k\in\Z_{\ge 0}$. If $k\ge\max\{0,N+1-s\}$, then $s+k\ge N+1$, and hence $u_{s+k}\notin V_{\prec}$. Therefore, $\pr^{\ip{r+k}}(w)\notin V_{\prec}$ for every $k\ge\max\{0,N+1-s\}$. That is, $\pr^{\ip{k}}(w)\notin V_{\prec}$ for every $k\ge r+\max\{0,N+1-s\}$.
    
        Hence, the set $F:=\{k\in\Z_{\ge 0}:\pr^{\ip{k}}(w)\in V_{\prec}\}$ is finite and nonempty. Let $k_0:=\max F$ and we define $\omega:=\pr^{\ip{k_0}}(w)$. Then $\omega\in V_{\prec}$ and $\pr^{\ip{k}}(\omega)\notin V_{\prec}$ for every $k\in\Z_{\ge 1}$. Since $\scrT$ is leafless, $\abs{\Chi(\pr^{\ip{k}}(\omega))}=1$ for every $k\in\Z_{\ge 1}$. Hence, $\omega$ is the generalised root of $\scrT$.
    \end{proof}

    The following proposition gives necessary structural properties of a directed tree which admits an analytic norm-increasing weighted shift with positive weights which fails to possess the wandering subspace property. Suppose first that $\scrT=(V,E)$ is rootless and leafless with $V_{\prec}=\varnothing$. Then $\scrT$ is isomorphic to the directed tree on $\Z$ given by $\Chi(n)=\{n+1\}$ for each $n\in\Z$. Hence, no weighted shift on $\scrT$ with positive weights can be analytic. Thus, in the case of positive weights, the implication from analyticity to the wandering subspace property is vacuous. We now consider the case which is relevant to \Cref{thm:counterexample}. 

    \begin{proposition}
        \label{prop:necessity-of-tree structure}
        Let $\scrT=(V,E)$ be a directed tree with $V_{\prec}\ne\varnothing$, and let $S_{\lambda}\in\calB(\ell^2(V))$ be an analytic norm-increasing weighted shift with positive weights which does not possess the wandering subspace property. Then $V$ is countable, $\scrT$ is rootless and leafless, $\scrT$ has infinite branching index, and every generation of $\scrT$ is infinite.
    \end{proposition}

    \begin{proof}
        By \cite[Proposition 3.1.10]{MR2919910}, we have that $\calV$ is countable. If $\scrT$ were rooted, $S_{\lambda}$ would possess the wandering subspace property (see \cite[Proposition 1.3.4]{MR3740250}). Hence $\scrT$ must be rootless. Moreover, if $\scrT$ had a leaf, $S_{\lambda}$ cannot be injective, and hence could not be norm-increasing. Thus, $\scrT$ must be leafless.

        If $\scrT$ had a finite branching index, then by \Cref{prop:characterisation-of-finite-branching-rootless-trees}, we have that $\scrT$ must have a generalised root. Then $S_{\lambda}$ cannot be analytic by \Cref{prop:generalised-root-not-analytic}. The same argument works if some generation of $\scrT$ were finite in view of \Cref{prop:some-finite-generation-implies-generalised-root}. This completes the proof. 
    \end{proof}

    \section{Analytic norm-increasing \texorpdfstring{$4$}{4}-isometries without the wandering subspace property}\label{sec:4-isometry}
    
    In this section, we construct analytic norm-increasing $4$-isometric weighted shifts on rootless directed trees that fail to possess the wandering subspace property. To this end, we consider the rootless quasi-Brownian tree $\scrT_{qb} = ( V,  E)$ of valency $2$ (see \cite[\S 5]{MR5052237}). Its vertex set is $ V = \Z_{\ge 0} \times \Z$, and its edge set is
    \begin{align*}
         E = &\left\{ \left( (0,m), (1, m) \right),  \left( (0,m), (0, m-1) \right) : m \in \Z \right\} \\ & \sqcup \left\{  \left( (n, m) , (n+1, m) \right) : n \ge 1, m \in \Z \right\}.
    \end{align*}
    Observe that $\scrT_{qb}$ meets all the necessary conditions mentioned in \Cref{prop:necessity-of-tree structure}. 
    See \Cref{fig:quasi-brownian-2} for an illustration of this tree.
    
    \begin{figure}[H]
        \centering  
        \includegraphics[width=0.7\textwidth]{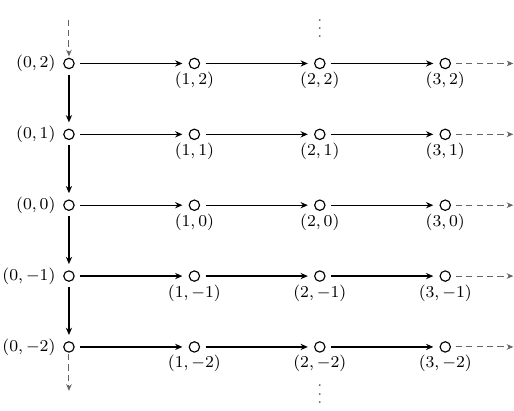}
        \caption{The quasi-Brownian tree of valency $2$}
        \label{fig:quasi-brownian-2}
    \end{figure}
    
    We use the following family of weighted shifts considered by S. Chavan and S. Trivedi in \cite[\S 5]{MR5052237}. Let $\{ p_m \}_{m \in \Z}$ be a bi-infinite sequence of quadratic polynomials with positive coefficients. We define a system of weights on $\scrT_{qb}$ given by
    \begin{equation}
        \label{eq:CS-weights}
        \lambda_{(n,m)}
        =
        \begin{dcases}
            \sqrt{\frac{m}{m+1}},
                & n=0,\quad m\ge 1, \\
            \frac{1}{\sqrt{m}},
                & n=1,\quad m\ge 1, \\
            1,
                & n\in\{0,1\},\quad m<1, \\
            \sqrt{\frac{p_m(n-1)}{p_m(n-2)}},
                & n\ge 2.
        \end{dcases}
    \end{equation}
    
    We state the following result of Chavan and Trivedi concerning the above weighted shifts.

    \begin{proposition}[{\cite[Proposition 5.1]{MR5052237}}]
        \label{prop:Chavan-Trivedi-analytic-not-wsp}
        Let $\{ p_m \}_{m \in \Z}$ be a bi-infinite sequence of quadratic polynomials with positive coefficients such that for all $m \in \Z$,
       \begin{equation}
        \begin{gathered}
            p_m(0) = 1, \\
            \sup_{m \in \mathbb{Z}}
            \left\{ p_m'(0),\, p_m''(0) \right\} < \infty, \\
            \inf_{m \in \mathbb{Z}} p_m''(0) > 0.
        \end{gathered}
        \label{eq:hypothesis-on-polynomials}
    \end{equation}
    Then the weighted shift $S_{\lambda}$ on $\ell ^2 ( V)$ associated with the weight sequence defined in \eqref{eq:CS-weights} is bounded, analytic, norm-increasing, and it does not have a Wold-type decomposition.
    \end{proposition}
    
    Since $S_{\lambda}$ is analytic, its hyper-range is trivial. Hence, for $S_{\lambda}$, the existence of Wold-type decomposition is equivalent to the wandering subspace property. Thus the weighted shift in the preceding proposition fails to possess the wandering subspace property.
    
    In the next result, we classify all weighted shifts $S_{\lambda}$ on $\ell ^2 ( V)$ which are $4$-isometries.
    
    \begin{theorem}
        \label{thm:classify-4-isometries}
        Let $p_m (t) =  1 + a_m t + b_m t^2, \quad m \in \Z$, where $a_m, b_m > 0$,
        \begin{equation}
            \sup_{m \in \Z} a_m < \infty, \quad \sup_{m \in \Z} b_m < \infty, \quad \inf_{m \in \Z} b_m > 0,
            \label{eqn:conditions-a-m-b-m}
        \end{equation}
        and let $S_\lambda$ be the weighted shift on $\scrT_{qb}$ with the system of weights in \eqref{eq:CS-weights}. Then the following are equivalent:
        \begin{enumerate}
            \item[(i)] $S_{\lambda}$ is a $4$-isometry;
            \item[(ii)] $\beta_4 (S_{\lambda}) e_{(0,m)} = 0$ for each $m \in \Z$;
            \item[(iii)] there exists $d \in \R$ such that for all $m \in \Z$,
            \begin{equation}
                \label{eq:4-isometry-characterisation}
                a_{m-1} - a_m + b_{m-1} + b_{m} = d + \delta_2 (m) - \delta_3 (m).
            \end{equation}
        \end{enumerate}
        Whenever either $(i)$, $(ii)$ or $(iii)$ holds, then $2 \inf_{m \in \Z} b_m \le d \le 2 \sup_{m \in \Z} b_m$, so $d > 0$. Moreover, $S_{\lambda}$ is bounded, analytic, and norm-increasing, does not possess the wandering subspace property, and is never $3$-concave. In particular, it is not a $3$-isometry, and hence is a strict 4-isometry.
    \end{theorem}
    
    \begin{proof}
        We first show that $S_{\lambda}$ is bounded, analytic, and norm-increasing and does not possess the wandering subspace property. Since $p_m ' (0) = a_m$ and $p_m '' (0) = 2 b_m$, the conditions \eqref{eqn:conditions-a-m-b-m} imply the conditions \eqref{eq:hypothesis-on-polynomials} in \Cref{prop:Chavan-Trivedi-analytic-not-wsp}. Hence, by \Cref{prop:Chavan-Trivedi-analytic-not-wsp}, $S_\lambda$ is bounded, analytic, norm-increasing and does not have the wandering subspace property.
    
        We proceed to prove the characterisation of $4$-isometries. For $n \in \Z_{\ge 1}$ and $j \in \Z_{\ge 0}$, $\Chi ^{\langle j \rangle} \left( (n,m) \right) = \{ (n+j, m) \}$. Therefore, by \Cref{lem:JJS-powers}, we have
        \begin{equation*}
            \norm{S_{\lambda}^{j} e_{(n,m)}} ^{2} = \prod_{k=1}^{j} \lambda_{(n+k, m)} ^{2} = \frac{p_m (n+j-1)}{p_m (n-1)}, \qquad n \in \Z_{\ge 1}, j \in \Z_{\ge 0}, m \in \Z.
        \end{equation*}
        We observe that the right hand side of the previous equality is a polynomial of degree at most $2$. Therefore, using \Cref{lem:compute-beta-m}, we have that $\beta_4 (S_{\lambda}) e_{(n,m)} = 0$ for every $n \in \Z_{\ge1}$ and $m \in \Z$. It follows that $S_\lambda$ is a $4$-isometry if and only if $\beta_4 (S_{\lambda}) e_{(0,m)} = 0$ for each $m \in \Z$.
    
        We first observe that for $j \in \Z_{\ge 1}$ and $m \in \Z$, $\Chi^{\langle j\rangle}((0,m)) = \Chi^{\langle j-1\rangle}((0,m-1)) \sqcup\{(j,m)\}.$ Using this equality and \Cref{lem:JJS-powers}, we have for $j \in \Z_{\ge 1}$ and $m \in \Z$,
            \begin{equation*}
                \begin{aligned}
                \norm{S_\lambda^j e_{(0,m)}}^2 &= \sum_{v\in\Chi^{\langle j-1\rangle}((0,m-1))} \lambda_{(0,m)\mid v}^{2} +\lambda_{(0,m)\mid(j,m)}^{\,2} \\
                &= \lambda_{(0,m-1)}^2 \sum_{v\in\Chi^{\langle j-1\rangle}((0,m-1))} \lambda_{(0,m-1)\mid v}^{\,2} + \prod_{k=1}^{j}\lambda_{(k,m)}^2 \\
                &= \lambda_{(0,m-1)}^2 \norm{S_\lambda^{j-1}e_{(0,m-1)}}^2 + \prod_{k=1}^{j}\lambda_{(k,m)}^2,
                \end{aligned}
            \end{equation*}
        and we also have
            \begin{equation*}
            \prod_{k=1}^{j}\lambda_{(k,m)}^2 = \lambda_{(1,m)}^2 \prod_{k=2}^{j}\frac{p_m(k-1)}{p_m(k-2)} = \lambda_{(1,m)}^2p_m(j-1).
            \end{equation*}
        Hence, for all $j \in \Z_{\ge 1}$ and $m \in \Z$, we have
        \begin{equation}
            \norm{S_\lambda^j e_{(0,m)}}^2 = \lambda_{(0,m-1)}^2 \norm{S_\lambda^{j-1}e_{(0,m-1)}}^2 + \lambda_{(1,m)}^2p_m(j-1).
            \label{eq:branch-rec}
        \end{equation}
    
        By using equation \eqref{eq:branch-rec} and the system of weights defined in equation \eqref{eq:CS-weights}, we have
        \begin{equation*}
            \norm{S_\lambda e_{(0,m)}}^2 =
            \begin{cases}
                1, & m\geq2,\\
                2, & m\leq1,
            \end{cases}
        \end{equation*}
        and
        \begin{equation*}
            \norm{S_\lambda^2 e_{(0,m)}}^2 =
                \begin{dcases}
                    \frac{m-1+p_m(1)}{m}, & m\geq3, \\
                    \frac{2+p_2(1)}{2}, & m=2, \\
                    2+p_m(1), & m\leq1.
                \end{dcases}
        \end{equation*}
        Similarly,
        \begin{equation*}
            \norm{S_\lambda^3 e_{(0,m)}}^2 =
                \begin{dcases}
                    \frac{m-2+p_{m-1}(1)+p_m(2)}{m}, & m\geq4,\\
                    \frac{2+p_2(1)+p_3(2)}{3}, & m=3,\\
                    \frac{2+p_1(1)+p_2(2)}{2}, & m=2,\\
                    2+p_{m-1}(1)+p_m(2), & m\leq1,
                \end{dcases}
        \end{equation*}
        and
        \begin{equation*}
        \norm{S_\lambda^4 e_{(0,m)}}^2 =
            \begin{dcases}
                \frac{m-3+p_{m-2}(1)+p_{m-1}(2)+p_m(3)}{m}, & m\geq5,\\
                \frac{2+p_2(1)+p_3(2)+p_4(3)}{4}, & m=4,\\
                \frac{2+p_1(1)+p_2(2)+p_3(3)}{3}, & m=3,\\
                \frac{2+p_0(1)+p_1(2)+p_2(3)}{2}, & m=2,\\
                2+p_{m-2}(1)+p_{m-1}(2)+p_m(3), & m\leq1.
            \end{dcases}
        \end{equation*}
    
        Since for $m \in \Z$, we have
        \begin{equation*}
            p_m(1)=1+a_m+b_m,\qquad
            p_m(2)=1+2a_m+4b_m,\qquad
            p_m(3)=1+3a_m+9b_m,
        \end{equation*}
        it follows by \Cref{lem:compute-beta-m} and the preceding computations that 
        \begin{equation*}
            \begin{aligned}
            \beta_4(S_\lambda)e_{(0,m)} =& \lambda_{(1,m)}^2 \big( a_{m-2}-2a_{m-1}+a_m+b_{m-2}-b_m\\
                                        &\hspace{24mm} +\delta_2(m)-2\delta_3(m)+\delta_4(m) \big)e_{(0,m)}.
            \end{aligned}
        \end{equation*}
    
        Since $\lambda_{(1,m)}>0$ for each $m \in \Z$, condition $(ii)$ of the statement of this theorem is equivalent to
            \begin{equation*}
                a_{m-2}-2a_{m-1}+a_m+b_{m-2}-b_m +\delta_2(m)-2\delta_3(m)+\delta_4(m)=0, \qquad m \in \Z.
            \end{equation*}
    
        The previous equation is equivalent to
            \begin{equation*}
                \begin{aligned}
                    &a_{m-1}-a_m+b_{m-1}+b_m-\delta_2(m)+\delta_3(m)\\
                    &\qquad = a_{m-2}-a_{m-1}+b_{m-2}+b_{m-1} -\delta_2(m-1)+\delta_3(m-1), \qquad m \in \Z,
                \end{aligned}
            \end{equation*}
        where we have used $\delta_2(m-1)=\delta_3(m)$ and $\delta_3(m-1)=\delta_4(m)$ for all $m \in \Z$.
    
        Hence, the map $m\in\Z\mapsto a_{m-1}-a_m+b_{m-1}+b_m-\delta_2(m)+\delta_3(m)$ is constant. Therefore, there exists $d\in\R$ such that
            \begin{equation*}
                a_{m-1}-a_m+b_{m-1}+b_m = d+\delta_2(m)-\delta_3(m),\qquad m\in\Z.
            \end{equation*}
        This completes the proof of the equivalences of $(ii)$ and $(iii)$, and hence, of $(i)$, $(ii)$ and $(iii)$.
    
        Now, suppose that one of the equivalent conditions $(i)$, $(ii)$ or $(iii)$ holds. Then by item $(iii)$, we have for $m \in \Z_{\ge 4}$ that $d= a_{m-1} - a_{m} + b_{m-1} + b_{m}$. Consequently, for $n \in \Z_{\ge 4}$, we have 
        \begin{equation*}
            (n-3) d = a_3 - a_{n} + \sum_{k=4}^{n} (b_{k-1} + b_k).
        \end{equation*}
        Dividing both sides by $n-3$ and using the fact that $\{ a_m \}_{m \in \Z}$ is bounded, we have by letting $n \to \infty$ that  $2\inf_{m\in\Z}b_m \le d \le 2\sup_{m\in\Z}b_m$. Consequently, $d > 0$.
        
        We proceed to show that $S_\lambda$ is never $3$-concave. Again, it follows by \Cref{lem:compute-beta-m} and the earlier computations that 
            \begin{equation*}
                \beta_3(S_\lambda)e_{(0,m)} = \lambda_{(1,m)}^2 \left( a_{m-1}-a_m+b_{m-1}+b_m -\delta_2(m)+\delta_3(m) \right)e_{(0,m)}.
            \end{equation*}
        If $S_\lambda$ were $3$-concave, then $\beta_3(S_\lambda)\leq0$. Consequently, for each $m \in \Z_{\ge 4}$, we have $a_{m-1}-a_m+b_{m-1}+b_m\le0$. Thus, $a_m-a_{m-1} \ge b_{m-1}+b_m \ge 2\inf_{k\in\Z}b_k$ for each $m \in \Z_{\ge 4}$. Taking sums, we have for $n \in \Z_{\ge 4}$, that $a_n \geq a_3+2(n-3)\inf_{k\in\Z}b_k$. This would contradict the fact that $\{ a_m \}_{m \in \Z}$ is bounded. Thus, $S_{\lambda}$ cannot be $3$-concave. In particular, it cannot be a $3$-isometry. Hence, whenever $(i), (ii)$ or $(iii)$ holds, $S_{\lambda}$ must be a strict $4$-isometry.
    \end{proof}

    In the following corollary, we classify all weighted shifts on $\scrT_{qb}$ with the system of weights as in \eqref{eq:CS-weights} where the sequence $\{ b_m \}_{m \in \Z}$ is a constant.

    \begin{corollary}
        \label{cor:classify-4-isometries}
        Let $p_m (t) =  1 + a_m t + b t^2, \; m \in \Z$, where $a_m, b > 0$, $\sup_{m \in \Z} a_m < \infty$ and let $S_\lambda$ be the weighted shift on $\scrT_{qb}$ with the system of weights in \eqref{eq:CS-weights}. Then the following are equivalent:
        \begin{enumerate}
            \item[(i)] $S_{\lambda}$ is a $4$-isometry;
            \item[(ii)] there exists $a > 1$ such that for all $m \in \Z$,
            \begin{equation*}
                a_m = \begin{cases}
                    a, & m \ne 2, \\
                    a-1, & m = 2.
                \end{cases}
                \label{eq:almost-constant-a}
            \end{equation*}
        \end{enumerate}
        Moreover, $S_{\lambda}$ is bounded, analytic, and norm-increasing, does not possess the wandering subspace property, and is never $3$-concave. In particular, it is not a $3$-isometry whenever (i) or (ii) holds, it is a strict $4$-isometry.
    \end{corollary}
    \begin{proof}
        If item (ii) holds, then \eqref{eq:4-isometry-characterisation} holds with $d=2b$. Hence, $S_\lambda$ is a $4$-isometry by \Cref{thm:classify-4-isometries}.

        Conversely, suppose $S_{\lambda}$ is a 4-isometry. Then we have that $2b \le d \le 2b$ by \Cref{thm:classify-4-isometries}. Hence, $d=2b$. Define $a := a_1$. Setting $m=2$ in  Equation \eqref{eq:4-isometry-characterisation}, we have that $a_2 = a-1$. Setting $m=3$, we get $a_3 = a$. For any $m \ne 2$, we have $a_m = a$.

        The other conclusions are immediate from \Cref{thm:classify-4-isometries}.
    \end{proof}

    \begin{example}
        Consider the following sequence of polynomials
        \begin{equation*}
            p_m (t) = \begin{cases}
                (1+t)^2, & m\ne 2, \\
                1+t+t^2, & m = 2.
            \end{cases}
        \end{equation*}
        By \Cref{cor:classify-4-isometries}, we have that the weighted shift $S_{\lambda}$ on $\scrT_{qb}$ with the system of weights in \eqref{eq:CS-weights} is a bounded, analytic, norm-increasing $4$-isometry, does not possess the wandering subspace property, and is never $3$-concave.
    \end{example}

    The aforementioned corollary proves \Cref{thm:counterexample} which we state here again for convenience.
    
    \mConcaveCounterexample*

\section{Analytic \texorpdfstring{$3$}{3}-isometries without the wandering subspace property}\label{sec:3-isometry}

    In this section, we construct a family of analytic $3$-isometries on the tree $\scrT_{qb}$ which do not possess the wandering subspace property.  We make the following remark which follows immediately in view of \Cref{prop:finite-codimension,prop:generalised-root-not-analytic}.
    
    \begin{remark}
        If $\scrT = (V,E)$ is a directed tree and $S_{\lambda}$ is an analytic weighted shift with positive weights and $\dim \ker S_{\lambda} ^{*} < \infty$ then $\scrT$ must be rooted.
        \label{remark:analytic-finite-codimension-rooted}
    \end{remark}
    
    Before we state and prove the next proposition, we state a result due to S. Chavan and S. Trivedi which gives a criteria when a weighted shift is analytic on a left-invertible rootless directed tree.  Given any rootless directed tree $\scrT = (V, E)$, we define for $v \in V$,
    \begin{equation*}
        A(v,n) := \begin{dcases}
            \{v\}, & n=0, \\
            \Chi^{\langle n\rangle} \left(\pr^{\langle n\rangle}(v)\right) \setminus \Chi^{\langle n-1\rangle} \left(\pr^{\langle n-1\rangle}(v)\right), & n\in\mathbb Z_{\geq 1}.
        \end{dcases}
    \end{equation*}
    and for a system $\lambda=\{\lambda_u\}_{u\in V}$ of positive weights, define
    \begin{equation*}
        \alpha_\lambda(v) := \sum_{n=0}^{\infty}\sum_{u\in A(v,n)} \left( \frac{\lambda_{\pr^{\langle n\rangle}(v)|u}} {\lambda_{\pr^{\langle n\rangle}(v)|v}} \right)^2.
    \end{equation*}

    \begin{proposition}[{\cite[Corollary 3.4]{MR5052237}}] \label{prop:3-isometry-characterisation}
        Let $\scrT = (V, E)$ be a leafless rootless directed tree, and let $S_{\lambda} \in \calB (\ell ^2 (V))$ be a left-invertible weighted shift with positive weights. Then $S_{\lambda}$ is analytic if and only if $\alpha_{\lambda} (v) = \infty$ for some, or equivalently, for every $v \in V$.
        \label{prop:analytic-criterion}
    \end{proposition}

    \begin{proposition}
        \label{prop:classify-3-isometries}
        Let $p (t) =  1 + a t + b t^2, \quad a, b >0$ and let $c,d >0$. Consider the weighted shift $S_{\lambda}$ on $\scrT_{qb}$ with the system of weights
            \begin{equation}
        \label{eq:CS-weights-3-isometry}
        \begin{gathered}
            \lambda_{(n,m)} = \begin{dcases}
                c, & n=0, \quad\quad m\in \Z, \\
                d, & n=1, \quad\quad m\in \Z, \\ 
                \sqrt{\frac{p(n-1)}{p(n-2)}}, & n \in \Z_{\ge 2},\quad m \in \Z.
            \end{dcases}
        \end{gathered}
    \end{equation}
    Then $S_{\lambda}$ is bounded and left-invertible. The following statements hold:
    \begin{enumerate}
        \item[\rm (i)] \begin{equation*}
            \inf_{v \in V} \norm{S_{\lambda} e_v}^{2} = \min \{ c^2 + d^2, 1 \}.
        \end{equation*}
        \item[\rm (ii)] $S_{\lambda}$ is norm-increasing if and only if $c^2 + d^2 \ge 1$.
        \item[\rm (iii)] $S_\lambda$ is analytic if and only if $c \le 1$.
        \item[\rm (iv)] $S_{\lambda}$ posseses the wandering subspace property if and only if $c^2 + d^2 > c$.
        \item[\rm (v)] $S_{\lambda}$ is a $3$-isometry if and only if
        \begin{equation*}
            d^2 ((1-c^2)a - (1+c^2)b) = (1-c^2)^{2} (c^2 + d^2 -1).
        \end{equation*}
        Moreover, whenever $S_\lambda$ is a $3$-isometry, it is strict $3$-isometry.
    \end{enumerate}
    \end{proposition}

    \begin{proof}
        Observe that 
        \begin{equation*}
            \norm{S_{\lambda} e_{(0,m)}}^{2} = c^2 +d ^2, \quad \norm{S_{\lambda} e_{(n,m)}} ^{2} = \frac{p(n)}{p(n-1)}, \qquad m \in \Z, n \in \Z_{\ge 1}.
        \end{equation*}
        Since $\frac{p(n)}{p(n-1)} \to 1$ as $n \to \infty$, $\sup_{(n,m) \in V} \norm{S_{\lambda} e_{(n,m)}}^2 < \infty$, hence $S_{\lambda}$ is bounded. Moreover, since $S_{\lambda} ^{*} S_{\lambda}$ is a diagonal operator with entries $\norm{S_{\lambda} e_{v}}^{2}, \quad v \in V$, we have that $\inf_{v \in V} \norm{S_{\lambda} e_v}^{2} = \min \{ c^2 + d^2, 1\} > 0$. Thus, $S_{\lambda}$ is left-invertible and norm-increasing if and only if $c^2 + d^2 \ge 1$. This proves {\rm (i)} and {\rm (ii)}.

        We proceed to prove {\rm (iii)}. By \Cref{prop:analytic-criterion}, it is necessary and sufficient to prove $\alpha_{\lambda} ((0,0)) = \infty$ in order to prove that $S_{\lambda}$ is analytic. By definitions, we have that $\pr ^{\langle n\rangle}((0,0))=(0,n)$ and $A((0,0),n)=\{(n,n)\}$ for every $n\in\mathbb Z_{\geq 1}$. Hence, $\lambda_{(0,n)|(n,n)}=d\sqrt{p(n-1)}$ and $\lambda_{(0,n)|(0,0)}=c^n$. Thus,
        \begin{equation*}
            \alpha_\lambda((0,0))=1+d^2\sum_{n=1}^{\infty}\frac{p(n-1)}{c^{2n}}.
        \end{equation*}
        It is easy to see that $\alpha_\lambda((0,0))=\infty$ if and only if $c\leq 1$. This completes the proof of item {\rm (iii)}.

        We proceed to prove item {\rm (iv)}. Using \Cref{prop:analytic-wsp-in-terms-of-cauchy-dual}, $S_{\lambda}$ possesses the wandering subspace property if and only if the Cauchy dual $S_{\lambda}'$ is analytic. Note that $S_{\lambda} '$ is a weighted shift on $\scrT_{qb}$ with positive weights whose weights are given by
        \begin{equation*}
            \lambda_{u} ' = \frac{\lambda_u}{\norm{S_{\lambda} e_{\pr (u)}}^{2}}, \qquad u \in V.
        \end{equation*}
        Therefore, we have that 
            \begin{equation*}
                \lambda'_{(n,m)} =\begin{dcases}
                    \frac{c}{c^2 + d^2}, & n=0, m\in \Z, \\
                    \frac{d}{c^2 + d^2}, & n=1, m \in \Z, \\
                    \sqrt{\frac{p(n-2)}{p(n-1)}}, & n \in \Z_{\ge 2}, m \in \Z.
                \end{dcases}
            \end{equation*}
        and, hence, 
        \begin{equation*}
            \lambda ' _{(0, n) | (0,0)} = \left( \frac{c}{c^2+d^2} \right)^{n}, \quad \lambda'_{(0,n) \mid (n,n)} = \frac{d}{(c^2 +d^2) \sqrt{p (n-1)}}, \qquad n \in \Z_{\ge 1}.
        \end{equation*}
        Observe that 
        \begin{equation*}
            \alpha_{\lambda'} ((0,0)) = 1 + \frac{d^2}{(c^2 + d^2)^2} \sum_{n=1}^{\infty} \left( \frac{(c^2 + d^2)^2}{c^2} \right)^{n} \frac{1}{p(n-1)}.
        \end{equation*}
        From here, it is easy to see that $\alpha_{\lambda'} ((0,0)) = \infty$ if and only if $c^2 + d^2 > c$. Thus, $S_{\lambda}$ possesses the wandering subspace property if and only if $c^2 + d^2 > c$. This proves item {\rm (iv)}.

        The proof of item {\rm (v)} is similar to that of \Cref{thm:classify-4-isometries} and follows from \Cref{lem:compute-beta-m}. We omit the proof.
    \end{proof}

    We remark that every norm-increasing weighted shift on $\scrT_{qb}$ with weights as in \eqref{eq:CS-weights-3-isometry} possesses the wandering subspace property. Indeed, if $S_{\lambda}$ failed to possess the wandering subspace property, then $c^2 + d^2 \le c$. Consequently, $c \in (0,1)$, and hence $c^2 + d^2 < 1$ which would imply that $S_{\lambda}$ is not norm-increasing. Moreover, no analytic weighted shift in this family can be cyclic. Indeed, if $S_{\lambda}$ were analytic and cyclic, then $\dim \ker S_{\lambda} ^{*} \le 1$. But then by \Cref{remark:analytic-finite-codimension-rooted}, $\scrT_{qb}$ must be rooted, which is a contradiction.

    \begin{example}
        In view of \Cref{prop:classify-3-isometries}, the weighted shift $S_{\lambda}$ on $\scrT_{qb}$ with the system of weights as in \eqref{eq:CS-weights-3-isometry} with $a=1, b= 3/2, c=d=1/2$ is an analytic $3$-isometry which does not possess the wandering subspace property, is not cyclic and is not norm-increasing. A quick computation shows that 
        \begin{equation*}
            \lambda_{(n,m)} = \begin{dcases}
                \frac{1}{2}, & n \in \{ 0,1 \}, m \in \Z \\
                \sqrt{\frac{3n^2 - 4n + 3}{3n^2 - 10n + 10}}, & n \in \Z_{\ge 2}, m \in \Z.
            \end{dcases}
        \end{equation*}
        \label{ex:analytic-non-cyclic-no-WSP}
    \end{example}

\section{Some open questions}\label{sec:last}
    
    For each $m\in\Z_{\ge 4}$, we constructed a counterexample to \Cref{qn:Shimorin} on the rootless quasi-Brownian directed tree of valency $2$, which is leafless, has infinite branching index, and has every generation infinite. In view of this and \Cref{prop:necessity-of-tree structure}, it is natural to ask the following question:
    
    \begin{question}
        Let $\scrT=(V,E)$ be a leafless rootless countable directed tree with infinite branching index such that every generation of $\scrT$ is infinite. Does there exist an analytic norm-increasing $4$-isometric weighted shift $S_\lambda\in\calB(\ell^2(V))$ with positive weights which does not possess the wandering subspace property?
    \end{question}
    
    We remark that the analogous question with $4$-isometric replaced by $3$-concave has a negative answer. We have seen by \Cref{thm:3-concave} that every norm-increasing $3$-concave weighted shift on a directed tree admits a Wold-type decomposition, and hence every such analytic weighted shift possesses the wandering subspace property.

    In view of \Cref{thm:counterexample}, Shimorin's \Cref{qn:Shimorin} is sharpened to the following question:
    
    \begin{question}
        Does every norm-increasing $3$-concave operator admit a Wold-type decomposition?
    \end{question}

\section*{Acknowledgments}
I am grateful to my advisor Md. Ramiz Reza for introducing me to the work of S. Shimorin \cite{zbMATH01572594} early in my doctoral studies. I would also like to thank Sameer Chavan and Shailesh Trivedi for their encouragement in the preparation of this article and for several valuable discussions.

\section*{AI Disclosure}
GPT-5.6 Sol by OpenAI was used to assist with computations, draw the TikZ figures, and proofread the manuscript. The author assumes full responsibility for the work.

\bibliographystyle{amsplain}
\bibliography{main}

\end{document}